\documentclass[a4paper,12pt]{article}
\usepackage[utf8]{inputenc}
\usepackage{fullpage}
\usepackage{amsmath, amssymb, amstext, amsfonts, amsthm,  array, enumerate,scalerel}
\usepackage{nicefrac}
\usepackage{xcolor}
\usepackage[english]{babel}
\usepackage{tikz}
\usepackage{stmaryrd}
\usepackage{tikz-3dplot}

\usetikzlibrary{patterns}
\usetikzlibrary{plotmarks}
\usepackage{enumitem}

\newtheorem{theorem}{Theorem}[section]
\newtheorem{lemma}[theorem]{Lemma}
\newtheorem{corollary}[theorem]{Corollary}
\newtheorem{proposition}[theorem]{Proposition}
\newtheorem{assumption}[theorem]{Assumption}

\theoremstyle{definition}
\newtheorem{definition}[theorem]{Definition}
\newtheorem{example}[theorem]{Example}

\theoremstyle{remark}
\newtheorem{remark}[theorem]{Remark}

\numberwithin{equation}{section}
\numberwithin{theorem}{section}

\newcommand{\E}{\mathbb{E}}

\newcommand{\Mcal}{{\mathcal M}}

\newcommand{\Dcal}{{\mathcal D}}

\newcommand{\Gcal}{{\mathcal G}}
\renewcommand{\P}{{\mathbb P}}

\newcommand{\p}{k}

\newcommand{\R}{\mathbb{R}}

\RequirePackage[colorlinks,citecolor=blue,urlcolor=blue, linkcolor=blue]{hyperref}

\newcommand\mydots{\hbox to 1em{.\hss.\hss.}}

\begin{document}

\title{Polynomial Volterra processes}

\author{Eduardo Abi Jaber\thanks{Ecole Polytechnique, CMAP, Palaiseau, France, eduardo.abi-jaber@polytechnique.edu} \quad   
Christa Cuchiero\thanks{Vienna University, Department of Statistics and Operations Research, Data Science @ Uni Vienna,  Wien, Austria, christa.cuchiero@univie.ac.at
} 
\quad
Luca Pelizzari\thanks{Technische Universität Berlin and Weierstrass Institut (WIAS), Berlin, Germany, pelizzari@wias-berlin.de}
\quad Sergio Pulido\thanks{Universit\'e Paris--Saclay, CNRS, ENSIIE, Univ \'Evry, Laboratoire de Math\'ematiques et Mod\'elisation d'\'Evry (LaMME), \'Evry, France, sergio.pulidonino@ensiie.fr} \\  Sara Svaluto-Ferro\thanks{University of Verona, Department of Economics,
		 Verona, Italy, sara.svalutoferro@univr.it   \newline
  This work was partially supported by the ``PHC AMADEUS'' programme (project number: 47561RJ), funded by the French Ministry for Europe and Foreign Affairs, the French Ministry for Higher Education and Research, and the Austrian Ministry for Higher Education.\newline
The first author is grateful for the financial support from the Chaires FiME-FDD and Financial Risks at \'Ecole Polytechnique. 
The research of the first and fourth authors benefited from the financial support of the Chaires Deep finance \& Statistics and Machine Learning \& systematic methods in finance at \'Ecole Polytechnique.
The second author gratefully acknowledges financial support through grant Y 1235 of the START-program and through the OEAD WTZ project FR 02/2022. The third author gratefully acknowledges funding by the Deutsche Forschungsgemeinschaft (DFG, German Research Foundation) under Germany’s Excellence Strategy – The Berlin Mathematics Research Center MATH+ (EXC-2046/1, project ID: 390685689). 
}}
\date{}

\maketitle

\begin{abstract}
We study the class of continuous polynomial Volterra processes, which we define as solutions to stochastic Volterra equations driven by a continuous semimartingale  with affine drift and quadratic diffusion matrix in the state of the Volterra process.  
To demonstrate the versatility of possible state spaces within our framework, we construct polynomial Volterra processes on the unit ball. 
This construction is based on a  stochastic invariance principle for stochastic Volterra equations with possibly singular kernels. 
Similarly to classical polynomial processes, polynomial Volterra processes allow for tractable expressions of the moments in terms of the unique solution to a system of deterministic integral equations, which reduce to a system of ODEs in the classical case.
By applying this observation to the moments of the finite-dimensional distributions we derive a uniqueness result for polynomial Volterra processes.
Moreover, we prove that the moments are polynomials with respect to the initial condition, another crucial property shared by classical polynomial processes.
The corresponding coefficients can be interpreted as a deterministic dual process and solve integral equations dual to those verified by the moments themselves.
Additionally, we obtain a representation of the moments in terms of a pure jump process with killing, which corresponds to another non-deterministic dual process.
\end{abstract}

\tableofcontents

 \section{Introduction}
Polynomial processes in finite dimensions, introduced in \cite{cuchiero2012polynomial} and \cite{filipovic2016polynomial}, constitute a class of time-homogeneous It\^o-semimartingales which are inherently tractable: conditional moments can be expressed through a \emph{deterministic dual process} which is the solution of a linear ODE. This is the so-called \emph{moment formula}. They form a rich class that includes Wright-Fisher diffusions (\cite{kimura1964diffusion}) from population genetics, Wishart correlation matrices (\cite{ahdida2013mean}), and affine processes (\cite{DFS:03}), just to name a few. Notably, polynomial diffusions offer greater flexibility than affine diffusions, accommodating more general semialgebraic state spaces, including in particular bounded state spaces; see \cite{filipovic2016polynomial} and \cite{LARSSON2017901} for a systematic analysis. The computational advantages due to the moment formula in the polynomial setting (see, e.g., \cite{ackerer2018jacobi} and \cite{filipovic2020markov}) have led to a wide range of applications, in particular in population genetics and mathematical finance. Indeed, in population genetics dual processes associated to moments and their interpretation in view of coalescent theory 
play an important role: the Wright-Fisher diffusion with seed-bank component (see, e.g., \cite{BCKW:16} and the references therein) is for instance an important example of a recently investigated two-dimensional polynomial process in this field.  
In mathematical finance, polynomial processes comprise a plethora of highly popular models, ranging from the famous Black Scholes model over certain jump-diffusions to Jacobi-type processes, which have been used for stochastic volatility models, life insurance liability modeling, variance swaps, and stochastic portfolio theory (see, e.g., \cite{ackerer2018jacobi,biagini2016polynomial,cuchiero2019polynomial,filipovic2016quadratic}). 

All these models share a finite dimensional Markov property which
sometimes may not be adequate, for instance, for modeling volatility   where path-dependence is crucial (see, e.g., \cite{cont2001empirical} and \cite{GL:23}).
This has motivated the emergence of numerous models in the literature based on stochastic Volterra equations, 
where the specification of the kernel offers greater flexibility to align with market data (\cite{abijaber2022joint, abi2024volatility,delemotte2023yet,  GJR:22,guyon2022does, GL:23, parent2022rough}). In particular, the singular fractional kernel is important in view of rough volatility models \cite{GJR:22}.
 To obtain models such as  the \emph{rough Heston model} \cite{ER:19},  the important class of affine processes has  been extended to the Volterra framework. In particular,  existence and uniqueness of solutions to the associated equations, invariance over certain domains, and formulas for the Fourier-Laplace transform have been established within the affine paradigm, see e.g., \cite{abi2019affine, cuchiero2020generalized, bondi2024affine}.
Note that Volterra-type processes are not only used in the realm of volatility modeling but also to model phenomena exhibiting short and long range dependence and
self-similarity. For instance, they have been applied in web-traffic \cite{N:94,Willi} and energy markets (see, e.g. \cite{BBV:13, BBV:14}). In the latter, so-called
Brownian semistationary processes, introduced in \cite{barndorff2009brownian}, as well as 
volatility modulated Volterra processes, first considered in \cite{barndorff2008time}, play an important role. All these processes can be embedded into the large class of ambit processes, pioneered in \cite{barndorff2007ambit} to
model turbulence and tumour growth. We also refer to the monograph \cite{barndorff2011ambit} for a far reaching analysis of these processes.

In analogy to continuous \emph{affine} Volterra processes as studied in \cite{abi2019affine}, we shall define continuous polynomial Volterra processes  as solutions $X$ of stochastic Volterra equations driven by a continuous semimartingale $Z$ depending on $X$ in a way that resembles the structure appearing in the classical framework. More precisely,  $Z$ has an affine drift and a quadratic diffusion matrix in $X$. To derive moment formulas in this setting, we draw inspiration from recent works on infinite-dimensional polynomial processes. While there is already a vast literature on finite dimensional polynomial processes, 
a systematic analysis of the infinite dimensional case was only recently provided in \cite{CLS:18, cuchiero2021infinite, cuchiero2021measure, benth2021independent}. The articles \cite{CLS:18, cuchiero2021measure} 
treat probability and non-negative measure-valued processes, which include the famous Fleming-Viot and the Dawson-Watanabe superprocess (see \cite{etheridge2000introduction} for an introduction to superprocesses and \cite{li2023measure} for measure-valued branching Markov processes). To accommodate these and also function space valued processes,  a common unifying framework that establishes in particular  the moment formula in a generic infinite dimensional setting has been built in \cite{cuchiero2021infinite}. Related concepts have also been in considered \cite{benth2021independent} and in \cite{benth2020abstract}.

Our work is the first systematic study extending the theory of polynomial processes to the Volterra setting and contributes to the existing literature on Volterra processes in multiple ways. In what follows, we describe the organization of the paper and our main contributions.

In Section \ref{sec:def_existence}, we set the stage by defining continuous polynomial Volterra processes and by recalling  moment estimates and existence of solutions to the associated stochastic Volterra equations from  \cite{abi2019affine}. To illustrate the versatility of possible state spaces, even in the Volterra case, we construct in Theorem~\ref{T:invarianceballpoly} the first non-trivial example of a polynomial Volterra process with possibly singular kernels that remains confined to the unit ball.   Notably, when restricted to the one dimensional case, our results provide the construction of Jacobi Volterra processes, see Corollary~\ref{cor:jacobi}. This construction relies on a more general result in Theorem~\ref{T:invarianceball} showing existence of general Volterra processes confined to the unit ball under structural assumption on the drift and diffusion coefficient of the driving semimartingale $Z$. 
Note that solutions to Volterra equations that remain within a convex set have so far been constructed  and studied for domains with no curvature, such as the non-negative orthant $\mathbb{R}^d_+$, see  \cite[Theorem 3.6]{abi2019affine} and the extensions in \cite{abi2021weak,alfonsi2023nonnegativity}. The primary challenge in constructing a solution $X$  that remains within a given convex set with curvature arises from the potential singularity of the kernel. This singularity could push the process $X$ outside the domain, if the volatility does not vanish on the boundary. This difficulty has already been observed in the construction of Volterra Wishart processes \cite{abi2022laplace,cuchiero2020generalized}.

Section \ref{sec:Moments} is dedicated to an analysis of the moments of polynomial Volterra processes. In particular, in Section \ref{sec:mainformula} we establish an extension of the moment formula for polynomial Volterra processes. This is the main moment formula in our work. It shows that the moments are the unique solutions to a system of deterministic integral equations,  which reduce to a system of ODEs for classical polynomial diffusions; see Theorem \ref{th:momentformula}.
In the terminology of \cite{cuchiero2021infinite} this formula corresponds to the bidual moment formula. In contrast to the generic infinite dimensional framework of \cite{cuchiero2021infinite} we can here actually prove existence and uniqueness of the system of deterministic integral equations.
Using a variation of constants technique, our arguments can be applied to deduce more explicit expressions for the first and second-order moments and for all moments in the affine case, see Section \ref{sec:varconstmoments}. Moreover, we elucidate in Section \ref{sec:momentsarepoly} a crucial structural property, shared also by classical polynomial diffusions, 
namely that the moments are polynomials with respect to the initial condition. The corresponding coefficients
can be interpreted as deterministic dual process and 
solve integral equations dual to those verified by the moments themselves; see Theorem \ref{thm:momentsarepolynomials}.
In the terminology of \cite{cuchiero2021infinite} this  corresponds to the dual moment formula, where 
we can again prove existence and uniqueness of the corresponding equations (which had to be assumed in the general framework of \cite{cuchiero2021infinite}).
Additionally, in Section \ref{sec:uniqueness}, we illustrate how our results and arguments can be applied to the moments of the finite-dimensional distributions. Our considerations then also lead to a novel result regarding the uniqueness in law for solutions to stochastic Volterra equations in the polynomial framework, as proved in Theorem \ref{thm:uniqueness}.

In Section \ref{sec:jumprep}, we get inspiration from the work on classical Flemming Viot processes and general infinite dimensional polynomial processes as considered in \cite{cuchiero2021infinite} to show that the moments of a polynomial Volterra process can be expressed in terms of expectations of a functional of a (finite dimensional) pure jump process with killing. Indeed, we consider a function valued lift to the so-called Filipovi\'c space \cite{filipovic2001consistency}, denoted by $B$, and 
then apply the `dual process' approach. Denote the infinitesimal generator of the function valued lift by $\mathcal{A}$  and consider  polynomials $f: B \times \mathbb{R}^k \to \mathbb{R},\,  (\lambda, x) \mapsto \lambda(x_1)\cdots \lambda(x_k)$. Then we show that in the case of bounded kernels there exists a $k$-dimensional pure jump process with killing, denoted by $U$, with generator $\mathcal{L}$  such that
\begin{align}\label{eq:dualgen}
\mathcal{A}f(\cdot, x)(\lambda)=\mathcal{L}f(\lambda, \cdot)(x).
\end{align}
Modulo several technical conditions, e.g., stated in~\cite[Lemma A.1]{cuchiero2023signature}, it then holds that
\begin{align}\label{eq:duality}
\mathbb{E}_{\lambda_0}[f(\lambda_t,x)] =\mathbb{E}_{Y_0=x}[f(\lambda_0, U_t)].
\end{align}
As the evaluation of $\lambda$ at $0$ corresponds to the Volterra process, we get a representation of the $k$th moment by setting $x=Y_0=0 \in \mathbb{R}^k$.  For the homogeneous case with linear drift and volatility this formula is rigorously proved in Proposition~\ref{prop:jump} and Proposition~\ref{prop:jumpmulti} for the multivariate case. Since for  $\lambda_0 \equiv X_0$ we have    $f(\lambda_0,x)\equiv X_0^k$, we also see that the $k$th moment of the Volterra process is a monomial of degree $k$ in the initial value, which is thus a special case of Theorem \ref{thm:momentsarepolynomials}. For the general non-homogeneous case we retrieve also a similar formula, see Remark~\ref{rem:nonhomogen}.
In this context, let us also mention that the results of Section \ref{sec:momentsarepoly} can be seen from a similar duality point of view, here with $\mathcal{L}$ corresponding to the dual operator in the terminology of \cite{cuchiero2021infinite},  giving rise to a system of deterministic PDEs.
From a numerical perspective 
 the jump representation can sometimes have advantages as it is easy to simulate from a pure jump process and then compute the right hand side of \eqref{eq:duality} via Monte Carlo.
 
Section \ref{sec:unitball} provides the proof of Theorem~\ref{T:invarianceball}.
Appendix \ref{app:inteq} contains the main results necessary to guarantee existence and uniqueness of solutions to the equations presented in Section \ref{sec:Moments}, namely the integral equations verified by the moments of polynomial Volterra processes and by the coefficients in the expression as polynomial with respect to the initial condition.

\textbf{Notation:} We denote polynomials on $\mathbb{R}^d$ of degree less than or equal to $n$ by $\mathrm{Pol}_n(\mathbb{R}^d)$. $\mathbb{N}$ is the set of natural numbers and $\mathbb{N}_0=\mathbb{N}\cup \{0\}$. For a multi-index $\alpha=(\alpha_1,\ldots,\alpha_d)$, with $\alpha_i\in \mathbb{N}_0$, we write $|\alpha|=\sum_{i=1}^d \alpha_i$ for the sum of its components. For a vector $x\in\mathbb{R}^m$, $x_i$ is its $i$-th coordinate. For a matrix $A\in\mathbb{R}^{d\times d}$, we denote by $|A|=\sqrt{\mathrm{Tr}(A^\top A)}$ its Frobenius norm. ${\mathrm{I}}_d$ is the $d\times d$ identity matrix. For any $t\geq 0$, we use the symbol $\mathbb{E}_t$ to denote the conditional expectation given the $\sigma$-algebra $\mathcal{F}_t$. We sometimes use the convolution notation $(f\ast g)(t)=\int_0^t f(t-s)g(s)ds$ for functions $f$ and $g$, and $(f\ast L)(t)=\int_0^t f(t-s)L(ds)$ for a measure $L$. 

\section{Definition and existence of polynomial Volterra processes}
\label{sec:def_existence}

Fix a dimension $d\in \mathbb{N}$ and consider a filtered probability space $(\Omega,\mathcal{F}, (\mathcal{F}_t)_{t\geq 0},\mathbb{P})$, where $(\mathcal{F}_t)_{t\geq 0}$ satisfies the usual conditions and $\mathcal F_0$ is the trivial $\sigma$-algebra on $\Omega$. A continuous polynomial Volterra process of convolution type is a $d$-dimensional adapted process $X$ with continuous trajectories solving a stochastic Volterra equation of the form
\begin{equation}\label{eq_CPVP}
    X_t = g_0(t) + \int_0^t K(t-s)b(X_s)ds + \int_0^t K(t-s)\sigma(X_s) dW_s,\quad t\ge 0,
\end{equation} 
where
\begin{itemize}
    \item $W$ is a $d$-dimensional Brownian motion,
    \item the {\em initial condition} $g_0\colon\mathbb{R}_+\to\mathbb{R}^d$ is in $C(\mathbb{R}_+,\mathbb{R}^d)$,
    \item the {\em convolution kernel} $K\colon\mathbb{R}_+\to\mathbb{R}^{d\times d}$ is in $L^2_{\mathrm{loc}}(\mathbb{R}_+,\mathbb{R}^{d\times d})$, 
    \item the map $b\colon\mathbb{R}^d\to\mathbb{R}^d$ has components in $\mathrm{Pol}_1(\mathbb{R}^d)$, and $\sigma\colon\mathbb{R}^d\to\mathbb{R}^{d\times d}$ is a continuous map such that $a(x)=\sigma(x)\sigma(x)^\top$ has entries in $\mathrm{Pol}_2(\mathbb{R}^d)$. More precisely,
    \begin{equation}\label{eq:fcts_b_a}
        b(x) = b_0+ \sum_{i=1}^d b_i x_i,\quad a(x) = A_0+\sum_{i=1}^d A_ix_i+\sum_{i,j=1}^d A_{ij}x_ix_j
    \end{equation}
    for some $b_i\in\mathbb{R}^d$ and $A_i,A_{ij}\in\mathbb{R}^{d\times d}$.
\end{itemize}
Observe that if all $A_{ij}=0$, $X$ is an affine Volterra process as in \cite{abi2019affine}.
As for stochastic differential equations, we speak of weak solutions to \eqref{eq_CPVP} whenever the filtered probability space and the underlying Brownian motion are not fixed a priori and they are part of the solution. In this case, with a slight abuse of terminology, we say that $X$ is a weak solution to \eqref{eq_CPVP}.

If we define the $d$-dimensional semimartingale $Z$ as
\begin{equation}\label{eq:PZ}
    Z_t = \int_0^t b(X_s)ds + \int_0^t \sigma(X_s) dW_s,
\end{equation}
then the stochastic Volterra equation \eqref{eq_CPVP} can be recast as
\[
    X_t = g_0(t) + \int_0^t K(t-s)dZ_s.
\]

The following proposition provides a priori estimates on the moments of any solution to \eqref{eq_CPVP}. 

\begin{proposition}\label{prop:estimatesmoments}
    Let $X$ be a continuous solution to \eqref{eq_CPVP}. Then, for any $p\in\mathbb{N}$ and $T\geq 0$,
    \begin{equation}
        \sup_{0\leq t\leq T}\mathbb{E}[|X_t|^p]\leq c
    \end{equation}
for some constant $c$ which depends only on $\sup_{0\leq t\leq T}|g_0(t)|$, $p$, $K|_{[0,T]}$, $b_i$, $A_i,A_{ij}$, and $T$.
\end{proposition}

\begin{proof}
    The initial curve $g_0$ is continuous and hence bounded on compacts. In addition, by \eqref{eq:fcts_b_a}, the drift function $b$ and the volatility function $\sigma$ are continuous and have linear growth. Therefore, to prove this result we can follow the same argument as in the proof of \cite[Lemma 3.1]{abi2019affine}.  
\end{proof}

The following theorem guarantees the existence of solutions to \eqref{eq_CPVP}. The next assumption is needed to state the result.

\begin{assumption}\label{ass:KcontinL2}
    There exists a constant $\gamma\in(0,2]$ such that $\int_0^h |K(t)|^2dt = O(h^\gamma)$ and $\int_0^T |K(t+h)-K(t)|^2 dt = O(h^\gamma)$ for every $T<\infty$.
\end{assumption}

\begin{theorem}[Existence of polynomial Volterra processes]
    Suppose that Assumption~\ref{ass:KcontinL2} holds. Then, for any $\alpha<\gamma/2$, \eqref{eq_CPVP} admits a weak solution $X$ such that $X-g_0$ has $\alpha$-H\"older continuous trajectories.

\end{theorem}

\begin{proof}
    This result is a consequence of \cite[Theorem 3.4]{abi2019affine} and its proof, which can be adapted to the framework of an initial continuous deterministic curve $g_0$ instead of a constant initial condition $X_0\in\mathbb{R}^d$. The main difference is that the proof presented in \cite{abi2019affine} relies on the existence of a resolvent of the first kind for the kernel $K$. This hypothesis is not necessary because the same tightness argument to construct the weak solution in \cite{abi2019affine} can be adapted by considering the integrated form of the equation \eqref{eq_CPVP}
    \begin{equation}\label{eq:CPVP_int}
        \int_0^t X_sds = \int_0^t g_0(s)ds+\int_0^t K(t-s)Z_sds
    \end{equation}
    with $Z$ as in \eqref{eq:PZ} instead of the form $L\ast (X-g_0)=Z$ with $L$ the resolvent of the first kind of $K$. The integrated form \eqref{eq:CPVP_int} of the stochastic Volterra equation \eqref{eq_CPVP}, which is suitable for stability results as illustrated in \cite[Section 3]{abi2021weak}, can be obtained using an argument based on the stochastic Fubini theorem as shown in \cite[Lemma 3.2]{abi2021weak}.
\end{proof}

\begin{example}
 Assumption \ref{ass:KcontinL2} is satisfied for the (possibly singular) fractional kernel $K(t) = t^{H-1/2}$ with $H \in (0,1)$. In particular, the corresponding polynomial Volterra  process \eqref{eq_CPVP} fails to be a  semimartingale whenever $H \neq 1/2$.    
\end{example}

\subsection{Polynomial Volterra processes on the unit ball}

In this section, we construct polynomial Volterra processes \eqref{eq_CPVP},
that remain in the unit ball of $\R^d$ defined by 
$$ \mathcal B = \{ x \in \R^d: x^\top x \leq 1  \}.$$

 For the rest of this section, we will assume that the kernel $K$ is scalar $K:\mathbb{R}_+ \rightarrow \mathbb{R}$, and we use the so-called {\em resolvent of the first kind}, which is a measure $L$ on $\R_+$ of locally bounded variation such that
\begin{equation} \label{res_L}
K*L = L*K \equiv 1,
\end{equation}
see \cite[Definition~5.5.1]{GLS:90}. Some examples of resolvents of the first kind are given in \cite[Table 1]{abi2019affine}. A resolvent of the first kind does not always exist. For the main result of this section, we need an additional assumption for the kernel.
 \begin{assumption}\label{ass:K orthant}
    $K$ is nonnegative, not identically zero, non-increasing and continuous on $(0,\infty)$, and its resolvent of the first kind $L$ is nonnegative and non-increasing in the sense that $s\mapsto L([s,s+t])$ is non-increasing for all $t\ge0$.
\end{assumption} 

\begin{example} \label{E:CM}
If $K$ is completely monotone on $(0,\infty)$ and not identically zero, then Assumption \ref{ass:K orthant} holds due to \cite[Theorem~5.5.4]{GLS:90}. Recall that a function $f$ is called completely monotone on $(0,\infty)$ if it is infinitely differentiable with $(-1)^k f^{(k)}(t) \ge0$ for all $t>0$ and $k\geq 0$. This covers, for instance, any constant positive kernel, the fractional kernel $t^{H-1/2}$ with $H\in(0,1/2]$, and the exponentially decaying kernel ${\rm e}^{-\beta t}$ with $\beta>0$. Moreover, sums and products of completely monotone functions are completely monotone.
\end{example} 

The next theorem provides the weak existence and uniqueness of a $\mathcal B$-valued polynomial Volterra process. The construction follows from a more general result for $\mathcal B$-valued stochastic Volterra equations given in Theorem~\ref{T:invarianceball}.   

\begin{theorem}[Existence and uniqueness of polynomial Volterra processes in the unit ball]\label{T:invarianceballpoly} Fix a scalar kernel $K:[0,T] \to \mathbb R$ that satisfies Assumptions \ref{ass:KcontinL2} and \ref{ass:K orthant}.
    Assume that $b$ and $\sigma$ are such that 
    $$   b(x) = b_0 + Bx, \quad \sigma (x) = c \sqrt{1- x^\top x} {\mathrm{I}}_d   1_{\{x\in \mathcal B\}}, \quad x \in \R^d, $$
    where ${\mathrm{I}}_d$ is the $d\times d$ identity matrix,  $c\in \R$, and $b_0\in \R^d$ and $B \in \R^{d\times d}$ are such that 
    \begin{align}\label{eq:exampleB}
        x^\top( b_0 + B x) \leq 0 \quad x\in \partial \mathcal B.
    \end{align} For any $X_0 \in \mathcal B$, there exists a unique continuous weak solution $X$ to \eqref{eq_CPVP} such that $X_t \in \mathcal B$ a.s. for every $t\in [0,T]$. 
\end{theorem}

\begin{proof}
    Straightfoward application of  Theorem~\ref{T:invarianceball} and Theorem \ref{thm:uniqueness} below. 
\end{proof}

By restricting to the one-dimensional case, Theorem~\ref{T:invarianceballpoly} allows the construction of Jacobi Volterra processes on general compact intervals as shown in the next corollary.

\begin{corollary}\label{cor:jacobi}
    Let $\alpha_1 \leq \alpha_2$, $ b \in [\alpha_1,\alpha_2]$, $\lambda\geq 0$ and $c>0$. Fix a scalar kernel $K:[0,T] \to \mathbb R$ that satisfies Assumptions \ref{ass:KcontinL2} and \ref{ass:K orthant}.    Then, there exists a   unique weak $[\alpha_1,\alpha_2]$-valued  solution to the equation 
    \begin{align}\label{eq:svejacobi}
         Y_t &= Y_0 + \lambda\int_0^t K(t-s) (b - Y_s) ds + c\int_0^t K(t-s) \sqrt{(Y_s-\alpha_1)(\alpha_2 - Y_s)}dW_s, \\
         \quad Y_0& \in [\alpha_1,\alpha_2],
    \end{align}
that we call Jacobi Volterra process on $[\alpha_1, \alpha_2]$.
\end{corollary}

\begin{proof}
    We note that in dimension $d=1$ we have that $\mathcal B= [-1,1]$. In this case, the specification of Example~\ref{eq:ballcoeff} yields an $[-1,1]$-valued  Volterra Jacobi process in the form 
    \begin{align*}
        &X_t = X_0 + \int_0^t K(t-s)\lambda(\tilde b -  X_s) ds + \int_0^t K(t-s)c \sqrt{(1-X_s)(1+ X_s)} dW_s,\\
        &X_0 \in [-1,1],
    \end{align*}
for any      $\tilde b\in [-1,1]$.
 Indeed, in this case it is straightforward to check that \eqref{eq:exampleB} is satisfied since
    \begin{align}
        x\lambda(\tilde b- x) = 
\begin{array}{cc}
\begin{cases} 
     \lambda  (\tilde b -1 ) \leq 0 & x=1 \\
       -1 \lambda  (\tilde b +1 ) \leq 0 & x=-1.
\end{cases}
\end{array}
    \end{align}
    Taking $\tilde b$ such that $b= \frac{\alpha_2 - \alpha_1} 2 \tilde b + \frac{\alpha_1 + \alpha_2}{2} $ we readily  get that    $$ Y = \frac{\alpha_2 - \alpha_1} 2 X + \frac{\alpha_1 + \alpha_2}{2}$$
is a   Jacobi Volterra process on $[\alpha_1,\alpha_2]$ that satisfies \eqref{eq:svejacobi}. 
\end{proof}

\begin{remark} In \cite{filipovic2016polynomial} the authors study stochastic invariance of polynomial diffusions, that is solutions to \eqref{eq_CPVP} for the trivial kernel $K \equiv 1$, for more general state-spaces $E \subseteq \mathbb{R}^d$. In the case of $E=\mathcal{B}$, choosing $Q={\mathrm{I}}_d$ and $c=0$ in \cite[Proposition 6.1]{filipovic2016polynomial} essentially corresponds to Theorem \ref{T:invarianceballpoly} for the trivial kernel. In contrast with the diffusion case, where a tagential diffusive behaviour of the volatility component $\sigma$ is possible on the boundary, here we restrict to vanishing volatility at the boundary. Because of the possible singularity of the kernel at $0$, we expect that any tangential diffusive behavior of the volatility $\sigma$ at the boundary will push the process outside the ball $\mathcal B$.   
\end{remark}

\section{Moments of polynomial Volterra processes}
\label{sec:Moments}

Our aim is to find moment formulas for a continuous Volterra process $X$ solving \eqref{eq_CPVP}. More precisely, we want to obtain formulas for expressions of the form 
\begin{equation}\label{eq:goalmoment}
    \mathbb{E}[X_t^{\alpha}] = \mathbb{E}[X_{1,t}^{\alpha_1}\cdots X_{d,t}^{\alpha_d}],\quad t\ge 0,
\end{equation}
with $\alpha =(\alpha_i)_{i=1}^d\in\mathbb{N}_0^d$ a multi-index. One of the difficulties to characterize these moments, compared to the classical framework when the kernel $K$ is equal to ${\mathrm{I}}_d$, stems from the fact that $X$ is not necessarily a Markovian semimartingale. To circumvent this complication, and inspired  by previous works such as \cite{abi2019markovian,bondi2024affine, jacquier2019deep, viens2019martingale}, for each $T\geq 0$, we consider the process (indexed in time by $t$)
\begin{equation}\label{eq:adjforward}
    g_t(T) = g_0(T) + \int_0^t K(T-s)dZ_s,\quad t\leq T,
\end{equation}
for $Z$ as in \eqref{eq:PZ}. The following lemma shows that we can control the moments of the process $g$.
\begin{lemma}
\label{lem:momentsg}
    Suppose that $X$ is a continuous process solving \eqref{eq_CPVP} and define the processes $g$ as in \eqref{eq:adjforward}. Then, for any $p\in\mathbb{N}$ and $0\leq T\leq T'$,
    \begin{equation}\label{eq:estimateG}
        \mathbb{E}\left[\sup_{0\leq t\leq T}|g_t(T)|^p\right]\leq c
    \end{equation}
for some constant $c$ which depends only on $\sup_{0\leq t\leq T'}|g_0(t)|$, $p$, $K|_{[0,T']}$, $b_i$, $A_i,A_{ij}$, and $T'$.
\end{lemma}

\begin{proof}
    It is sufficient to prove the inequality \eqref{eq:estimateG} for $p\geq 2$. Given $T\ge 0$, thanks to the Burkholder-Davis-Gundy and Jensen's inequalities
    \begin{equation}\label{eq:BDG_processG}
    \begin{split}
        \mathbb{E}\left[\sup_{0\leq t\leq T}\left|\int_0^t K(T-s)\sigma(X_s)dW_s\right|^p\right]&\leq C\, \mathbb{E}\left[\left(\int_0^T |K(T-s)|^2|a(X_s)|ds\right)^{
        \frac{p}{2}
        }\right]\\
        &\leq C \left(\int_0^T|K(t)|^2dt\right)^{\frac{p}{2}}\sup_{0\leq t\leq T}\mathbb{E}[|a(X_t)|^{\frac{p}{2}}]\\
        &\leq C \left(\int_0^{T'}|K(t)|^2dt\right)^{\frac{p}{2}}\sup_{0\leq t\leq T'}\mathbb{E}[|a(X_t)|^{\frac{p}{2}}]
    \end{split}
    \end{equation}
    for some constant $C>0$. Similarly, multiple applications of Jensen's inequality yield
    \begin{equation}\label{eq:Jensen_driftG}
        \mathbb{E}\left[\left(\sup_{0\leq t\leq T}\left|\int_0^t K(T-s)b(X_s)ds\right|\right)^p\right]\leq (T')^{\frac{p}{2}}\left(\int_0^{T'}|K(t)|^2dt\right)^{\frac{p}{2}}\sup_{0\leq t\leq T'}\mathbb{E}[|b(X_t)|^{p}].
    \end{equation}
    Since the functions $b$ and $a$ have the form \eqref{eq:fcts_b_a}, \eqref{eq:BDG_processG} and \eqref{eq:Jensen_driftG} together with Proposition~\ref{prop:estimatesmoments} yield \eqref{eq:estimateG}.
\end{proof} 

Notice that $g_T(T)=X_T$ and more generally
\begin{equation}\label{eq:altfmlalift}
        g_t(T) = \mathbb{E}_t\left[X_T-\int_t^T K(T-s)b(X_s)ds\right],\quad t\leq T,
\end{equation}
because -- as a result of the proof of Lemma \ref{lem:momentsg} above and in particular \eqref{eq:BDG_processG} --  the process $M_t = \int_0^t K(T-s)\sigma(X_s)dW_s$, $t\leq T$, is a martingale. Moreover, for each $T\geq 0$, $g_t(T)$ is a semimartingale with dynamics
\begin{equation}\label{eq:dynG}
    dg_t(T) = K(T-t)dZ_t = K(T-t)b(X_t)dt+ K(T-t)\sigma(X_t)dW_t,
\quad t < T.
\end{equation} To study the moments of $X$ in \eqref{eq:goalmoment}, we need to understand the behavior of more general moments of the (infinite dimensional) processes $g$ defined in \eqref{eq:adjforward}. To this end, we consider expressions of the form
\begin{equation}\label{eq:momentslift}
    \mathfrak{m}^{(p)}(t,T_1,\ldots,T_p;w) = \mathbb{E}\left[\prod_{n=1}^p g_{i_n,t}(T_n) \right]
\end{equation}
where $p\in\mathbb{N}$, $0\leq t\leq \min\{T_1,\ldots,T_p\}$, $w=(i_n)_{n=1}^p\in\{1,\ldots,d\}^p$, and $g_{i_n,t}(T_n)$ is the $i_n$-th coordinate of $g_t(T_n)$. If $d=1$, we can omit the argument $w$ and write $\mathfrak{m}^{(p)}(t,T_1,\ldots,T_p)$ for $p\in\mathbb{N}$. We shall use the convention $\mathfrak{m}^{(0)}\equiv 1$.

Notice that
\begin{equation}\label{eq:moments_f_vs_X}
    \mathfrak{m}^{(p)}(t,t,\ldots,t;w) = \mathbb{E}[X_t^{\alpha(w)}]
\end{equation}
where $\alpha(w)$ is the multi-index given by $\alpha_k(w) = \#\{n: i_n=k\}$, $k=1,\ldots,d$. In particular, $|\alpha(w)|=p$.  

\subsection{The main moment formula}\label{sec:mainformula}

The main result of this section is a characterization for the functions $\mathfrak{m}$ defined in \eqref{eq:momentslift}, which in view of \eqref{eq:moments_f_vs_X} determine the moments. More precisely, we will see in Theorem~\ref{th:momentformula} below, that for some fixed level $N$, the vector-valued function \begin{equation}\label{eq:momentfunction}
    (t,T_1,\dots,T_N) \mapsto \{\mathfrak{m}^{(p)}(t,T_1,\dots,T_p;w): p\in\{0,\dots,N\} \text{ and }w\in \{1,\ldots,d\}^p \}
\end{equation} is the unique continuous solution to a specific integral equation. 
Referring to the terminology of \cite{cuchiero2021infinite}, this thus gives an existence and uniqueness result for the bidual moment formula (which had to be assumed in the generic infinite dimensional setting of \cite{cuchiero2021infinite}).
In order to give a precise statement, we have to fix some notations.
Given $N\in\mathbb{N}$, we define the set
\begin{equation}\label{eq:setI_N}
    \mathcal{I}^{(N)} = \{(p,w) : p\in\{0,1,\ldots,N\} \text{ and }w\in \{1,\ldots,d\}^p\}.
 \end{equation}
Let $D_N$ be the cardinality of the set $\mathcal{I}^{(N)}$. Notice that $D_N=N+1$ for $d=1$ and $D_N = \sum_{p=0}^N d^p = (d^{N+1}-1)/(d-1)$ for $d>1$. For $T\ge 0$, let 
\begin{equation}\label{eq:timedomain}   
\mathcal{D}^{(N)}_T = \{(t,T_1,\ldots,T_{N})\in [0,T]^{N+1}:t\leq \min\{T_1,\ldots,T_N\}\},  
\end{equation}
and let $\pi$ be an enumeration of $\mathcal{I}^{(N)}$. As we will see in Theorem \ref{th:momentformula} below, the function in \eqref{eq:momentfunction} belongs to the following space.
\begin{definition}\label{eq:spacechiN}
Let $\mathcal{X}^{(N)}_T$ be the space of $\mathbb{R}^{D_N}$-valued bounded functions $\mathbf{f}$ on $\mathcal{D}_T^{(N)}$ such that the $\pi(p,w)$-th component $\mathbf{f}_{\pi(p,w)}$ only depends on the variables $(t,T_1,\dots,T_p)$ for any $(p,w)\in \mathcal{I}^{(N)}$, and $\mathbf{f}_{\pi(0,\emptyset)}\equiv 1$.
\end{definition}
Given $\mathbf{f}\in \mathcal{X}^{(N)}_T$, define the function $\mathcal{M}^{(N)}_T\mathbf{f}$ on $\mathcal{D}_T^{(N)}$ as follows: $(\mathcal{M}^{(N)}_T\mathbf{f})_{\pi(0,\emptyset)}\equiv 0$, and for $(p,w)\in \mathcal{I}^{(N)}$ with $1\leq p\leq N$ and $w=(i_n)_{n=1}^p$,
\begin{equation}\label{eq:operatorM}
\begin{split}
 &(\mathcal{M}^{(N)}_T\mathbf{f})_{\pi(p,w)}(t,T_1,\ldots,T_{N})=
    \sum_{n=1}^p \int_0^t e_{i_n}^\top K(T_n-r)b_0 \mathbf{f}_{\pi(p-1,w_{-n})}(r,(T_m)_{m\neq n})dr\\
    &\quad+\sum_{n=1}^p\sum_{j=1}^d  \int_0^t e_{i_n}^\top K(T_n-r)b_j \mathbf{f}_{\pi(p,w_{-n}^j)}(r,r,(T_m)_{m\neq n})dr\\
    &\quad+\sum_{1\leq n<m\leq p}\int_0^t  e_{i_n}^\top K(T_n-r)A_0K(T_m-r)^\top e_{i_m}\mathbf{f}_{\pi(p-2,w_{-n,-m})}(r,(T_l)_{l\neq m,n})dr\\
    &\quad+\sum_{1\leq n<m\leq p}\sum_{j=1}^d \int_0^t e_{i_n}^\top K(T_n-r)A_jK(T_m-r)^\top e_{i_m}\mathbf{f}_{\pi(p-1,w_{-n,-m}^j)}(r,r,(T_l)_{l\neq m,n})dr\\
    &\quad+\sum_{1\leq n<m\leq p}\sum_{j,k=1}^d \int_0^t e_{i_n}^\top K(T_n-r)A_{jk}K(T_m-r)^\top e_{i_m}\mathbf{f}_{\pi(p,w_{-n,-m}^{j,k})}(r,r,r,(T_l)_{l\neq m,n})dr.
\end{split}
\end{equation}
In \eqref{eq:operatorM}, $e_i$ is the $i$-th canonical vector in $\mathbb{R}^d$, $w_{-n}$ is the vector obtained by erasing the $n$-th coordinate of $w$, $w_{-n,-m}$ is the vector obtained by erasing the $n$-th and $m$-th coordinates of $w$, $w_{-n}^j$ is the vector whose first coordinate is $j$ and the other coordinates are given by $w_{-n}$, and $w_{-n,-m}^{j,k}$ is the vector whose first two coordinates are $j,k$ and the other coordinates are given by $w_{-n,-m}$. Notice that thanks to the local square-integrability of $K$ the right side of \eqref{eq:operatorM} is well-defined for $\mathbf{f}\in\mathcal{X}_T^{(N)}$ and, moreover, $\mathcal{M}_T^{(N)}\mathbf{f}\in\mathcal{X}_T^{(N)}$.

The following theorem provides formulas for the moments defined in \eqref{eq:momentslift}. In view of \eqref{eq:moments_f_vs_X}, this result also establishes relations between the moments of $X$ in \eqref{eq:goalmoment} and the moments of the process $g$.

\begin{theorem}[Main moment formula]\label{th:momentformula}
Fix $N\in\mathbb{N}$ and $T\geq 0$. Define $\mathbf{m}:\mathcal{D}_T^{(N)}\to\mathbb{R}^{D_N}$ by $\mathbf{m}_{\pi(0,\emptyset)}=\mathfrak{m}^{(0)}\equiv 1$, and for $(p,w)\in\mathcal{I}^{(N)}$ with $1\leq p\leq N$, $\mathbf{m}_{\pi(p,w)}(t,T_1,\ldots,T_N)=\mathfrak{m}^{(p)}(t,T_1,\ldots,T_p;w)$ as in \eqref{eq:momentslift}. Then $\mathbf{m}\in\mathcal{X}_T^{(N)}$ (see Definition \ref{eq:spacechiN}) and $\mathbf{m}$ solves the integral equation
\begin{equation}\label{eq:momentformula}
    \mathbf{m}(t,T_1,\ldots,T_N) = \mathbf{m}(0,T_1,\ldots,T_N) + (\mathcal{M}_T^{(N)}\mathbf{m})(t,T_1,\ldots,T_N),\quad (t,T_1,\ldots,T_N)\in\mathcal{D}_T^{(N)},
\end{equation}
with $\mathcal{M}_T^{(N)}$ as in \eqref{eq:operatorM}. Furthermore, $\mathbf{m}$ is the unique solution in $\mathcal{X}_T^{(N)}$ of \eqref{eq:momentformula} with initial condition $\mathbf{m}(0,T_1,\ldots,T_N)$, and $\mathbf{m}$ is continuous on $\mathcal{D}_T^{(N)}$.
\end{theorem}

\begin{proof}
To prove that $\mathbf{m}\in\mathcal{X}_T^{(N)}$ it is enough to show that $\mathbf{m}$ is bounded on $\mathcal{D}_T^{(N)}$. This is a consequence of H\"older's inequality and Lemma \ref{lem:momentsg}. Indeed, by Lemma \ref{lem:momentsg}, for any $i_n\in\{1,\ldots,d\}$, $p\in\mathbb{N}$, and $t\leq T_n\leq T$,
\[
    \mathbb{E}[|g_{i_n,t}(T_n)|^p]\leq \mathbb{E}\left[\sup_{0\leq t\leq T_n}|g_{i_n,t}(T_n)|^p\right]\leq c
\]
where $c$ is constant depending only on $\sup_{0\leq t\leq T}|g_0(t)|$, $p$, $K|_{[0,T]}$, $b_i$, $A_i,A_{ij}$, and $T$. 

We now prove \eqref{eq:momentformula}. Clearly this equation holds by definition over the coordinate $\pi(0,\emptyset)$. We consider then $(p,w)\in\mathcal{I}^{(N)}$ such that $1\leq p\leq N$ and $w=(i_n)_{n=1}^p$. It\^o's formula, together with \eqref{eq:dynG}, yields
\begin{equation}\label{eq:Ito_prod_lift}
\begin{split}
    d\left(\prod_{n=1}^p g_{i_n,r}(T_n)\right) &= \sum_{n=1}^p \left(\prod_{m\neq n} g_{i_m,r}(T_m)\right)dg_{i_n,r}(T_n)\\
    &\quad + \sum_{1\leq n<m\leq p}\left(\prod_{l\neq n,m} g_{i_l,r}(T_l)\right)d\langle g_{i_n,\cdot}(T_n), g_{i_m,\cdot}(T_m)\rangle_r.
\end{split}
\end{equation}

The local martingale part in \eqref{eq:Ito_prod_lift} is a true martingale thanks to \eqref{eq:dynG}, \eqref{eq:fcts_b_a}, the fact that $X_{r}=g_{r}(r)$, and Lemma \ref{lem:momentsg}. The finite variation part can be written as
\begin{align*}
    &\sum_{n=1}^p  e_{i_n}^\top K(T_n-r)b_0\left(\prod_{m\neq n} g_{i_m,r}(T_m)\right)dr\\
    &+\sum_{n=1}^p\sum_{j=1}^d e_{i_n}^\top K(T_n-r)b_jX_{j,r}\left(\prod_{m\neq n} g_{i_m,r}(T_m)\right)dr\\
    &+\sum_{1\leq n<m\leq p}e_{i_n}^\top K(T_n-r)A_0K(T_m-r)^\top e_{i_m}\left(\prod_{l\neq m,n} g_{i_l,r}(T_l)\right)dr\\
    &+\sum_{1\leq n<m\leq p}\sum_{j=1}^d e_{i_n}^\top K(T_n-r)A_jK(T_m-r)^\top e_{i_m}X_{j,r}\left(\prod_{l\neq m,n} g_{i_l,r}(T_l)\right)dr\\
    &+\sum_{1\leq n<m\leq p}\sum_{j,k=1}^d e_{i_n}^\top K(T_n-r)A_{jk}K(T_m-r)^\top e_{i_m}X_{j,r}X_{k,r}\left(\prod_{l\neq m,n} g_{i_l,r}(T_l)\right)dr.
\end{align*}

Since $g_{i_n,r}(r)=X_{i_n,r}$, integrating on $[0,t]$, taking expectation, using Fubini's theorem -- which can be applied thanks to Lemma \ref{lem:momentsg} -- and by the definition of $\mathfrak{m}$ in \eqref{eq:momentslift}, we obtain
\begin{equation}\label{eq:momentformulalong}
\begin{split}
    &\mathfrak{m}^{(p)}(t,T_1,\ldots,T_p;w)=\prod_{n=1}^p g_{i_n,0}(T_n)\\
    &\quad+\sum_{n=1}^p \int_0^t e_{i_n}^\top K(T_n-r)b_0 \mathfrak{m}^{(p-1)}(r,(T_m)_{m\neq n};w_{-n})dr\\
    &\quad+\sum_{n=1}^p\sum_{j=1}^d \int_0^t e_{i_n}^\top K(T_n-r)b_j \mathfrak{m}^{(p)}(r,r,(T_m)_{m\neq n};w_{-n}^j)dr\\
    &\quad+\sum_{1\leq n<m\leq p}\int_0^t e_{i_n}^\top K(T_n-r)A_0K(T_m-r)^\top e_{i_m}\mathfrak{m}^{(p-2)}(r,(T_l)_{l\neq m,n};w_{-n,-m})dr\\
    &\quad+\sum_{1\leq n<m\leq p}\sum_{j=1}^d\int_0^t e_{i_n}^\top K(T_n-r)A_jK(T_m-r)^\top e_{i_m}\mathfrak{m}^{(p-1)}(r,r,(T_l)_{l\neq m,n};w_{-n,-m}^j)dr\\
    &\quad+\sum_{1\leq n<m\leq p}\sum_{j,k=1}^d \int_0^t e_{i_n}^\top K(T_n-r)A_{jk}K(T_m-r)^\top e_{i_m}\mathfrak{m}^{(p)}(r,r,r,(T_l)_{l\neq m,n};w_{-n,-m}^{j, k})dr.
\end{split}
\end{equation}
This is precisely the coordinate $\pi(p,w)$ of the identity \eqref{eq:momentformula}, which concludes the proof of \eqref{eq:momentformula}. The uniqueness of solutions in $\mathcal{X}_T^{(N)}$ of the integral equation \eqref{eq:momentformula} and the fact that $\mathbf{m}$ is continuous on $\mathcal{D}_T^{(N)}$ are a consequence of Corollary \ref{cor:existenceumoments} in Appendix \ref{app:inteq}.  
\end{proof}

\begin{remark}\label{rmk:musiela}
If we define the function
\begin{equation}\label{eq:momentsMusiela}
       \overline{\mathbf{m}}(t,x_1,\ldots,x_N) = \mathbf{m}(t,t+x_1,\ldots,t+x_N),  
\end{equation}
for $t,x_1,\ldots,x_N\geq 0$ such that $t+x_i\leq T$, $i=1,\ldots,N$. Then, Theorem \ref{th:momentformula} implies that the function $\overline{\mathbf{m}}$ is a mild-solution to the following non local PDE
\begin{equation}\label{eq:momentPDE}
\begin{split}
    &\partial_t \overline{\mathbf{m}}_{\pi(p,w)}(t,x_1,\ldots,x_p)=(\partial_{x_1}+\cdots+\partial_{x_p})\overline{\mathbf{m}}_{\pi(p,w)}(t,x_1,\ldots,x_p)\\
    &\quad+\sum_{n=1}^p e_{i_n}^\top K(x_n)b_0 \overline{\mathbf{m}}_{\pi(p-1,w_{-n})}(t,(x_m)_{m\neq n})\\
    &\quad+\sum_{n=1}^p\sum_{j=1}^d e_{i_n}^\top K(x_n)b_j \overline{\mathbf{m}}_{\pi(p,w_{-n}^j)}(t,0,(x_m)_{m\neq n})\\
    &\quad+\sum_{1\leq n<m\leq p} e_{i_n}^\top K(x_n)A_0K(x_m)^\top e_{i_m}\overline{\mathbf{m}}_{\pi(p-2,w_{-n,-m})}(t,(x_l)_{l\neq m,n})\\
    &\quad+\sum_{1\leq n<m\leq p}\sum_{j=1}^d e_{i_n}^\top K(x_n)A_jK(x_m)^\top e_{i_m} \overline{\mathbf{m}}_{\pi(p-1,w_{-n,-m}^j)}(t,0,(x_l)_{l\neq m,n})\\
    &\quad+\sum_{1\leq n<m\leq p}\sum_{j,k=1}^d e_{i_n}^\top K(x_n)A_{jk}K(x_m)^\top e_{i_m} 
    \overline{\mathbf{m}}_{\pi(p,w_{-n,-m}^{j,k})}(t,0,0,(x_l)_{l\neq m,n}).
\end{split}
\end{equation}
Indeed, this can be deduced using the change of variables $T_n=t+x_n$ in \eqref{eq:momentformulalong} together with the definition of the function $\overline{\mathbf{m}}$ in \eqref{eq:momentsMusiela}.
\end{remark}

\begin{remark}
\label{rem:linkclassicalmomentfmla}
In the classical polynomial processes framework, where $K={\mathrm{I}}_d$ and $g_0\equiv X_0\in\mathbb{R}^d$, the process $g$ defined in \eqref{eq:adjforward} coincides with $X$, i.e. $g_t(T)=X_t$. Hence, in this case, the function $\mathbf{m}$ in Theorem \ref{th:momentformula} does not depend on $T_1,\ldots,T_N$. Consequently, the function $\overline{\mathbf{m}}$ defined in \eqref{eq:momentsMusiela} does not depend on $x_1,\ldots,x_N$ and the PDE \eqref{eq:momentPDE} reduces to a linear ODE with constant coefficients. When $d>1$, the dimension of this linear ODE -- which is $D_N=(d^{N+1}-d)/(d-1)$ in the framework of this study -- can be reduced to $\binom{N+d}{N}$. This dimension reduction is possible because, as it can be seen from \eqref{eq:moments_f_vs_X}, multiple $w\in\{1,\ldots,d\}^p$, $1\leq p\leq N$, yield the same moments of $X$. The solution to this linear ODE can be expressed in terms of an exponential matrix. This observation establishes a relation, when $K={\mathrm{I}}_d$, between Theorem \ref{th:momentformula} and the classical moment formula for polynomial processes as stated in \cite[Theorem 3.1]{filipovic2016polynomial}.  
\end{remark}

In the next subsection, using Theorem \ref{th:momentformula}, we establish an alternative moment formula using a variation of constants technique. This moment formula is useful to characterize the first and second order moments of a polynomial Volterra process, and the moments of affine Volterra processes, namely when $A_{ij}=0$ for all $i,j$ in \eqref{eq:fcts_b_a}.

\subsection{Moment formula using a variation of constants technique}\label{sec:varconstmoments}
Let $B$ be a matrix with columns equal to $b_1,\ldots,b_d$ and let $R_B$ be the resolvent of $-KB$, i.e. the solution to the linear equation $KB\ast R_B=KB+R_B$. Thanks to \cite[Lemma 2.5]{abi2019affine},  \eqref{eq_CPVP} is equivalent to the integral equation
\begin{equation*}
    X_t = g_0(t) - \int_0^t R_B(t-s)g_0(s)ds + \left(\int_0^t E_B(s)ds\right)b_0+ \int_0^t E_B(t-s)\sigma(X_s)dW_s
\end{equation*}
where $E_B = K-R_B\ast K$. Applying Theorem \ref{th:momentformula} to this reformulation of the stochastic Volterra equation yields alternative moment formulas. Indeed, let
\begin{equation}\label{eq:adjfwdvarconst}
    \widetilde{g}_t(T) = \widetilde{g}_0(T) + \int_0^t E_B(T-s)\sigma(X_s)dW_s = \mathbb{E}_t[X_{T}],
    \quad t\leq T,
\end{equation}
where $\widetilde{g}_0(t)= g_0(t) - \int_0^t R_B(t-s)g_0(s)ds + \left(\int_0^t E_B(s)ds\right)b_0 $. Define further, for $(p,w)\in\mathcal{I}^{(N)}$, with $1\leq p\leq N$ and $w=(i_n)_{n=1}^p$,
\begin{equation*}
    \widetilde{\mathfrak{m}}^{(p)}(t,T_1,\ldots,T_p;w) = \mathbb{E}\left[\prod_{n=1}^p \widetilde{g}_{i_n,t}(T_n) \right].
\end{equation*}
Adopt as before the convention $\widetilde{\mathfrak{m}}^{(0)}\equiv 1$. Theorem \ref{th:momentformula} shows that for $(p,w)\in\mathcal{I}^{(N)}$, with $1\leq p\leq N$ and $w=(i_n)_{n=1}^p$,
\begin{equation}\label{eq:momentfmlavarconstant}
\begin{split}
    &\widetilde{\mathfrak{m}}^{(p)}(t,T_1,\ldots,T_p;w)=\prod_{n=1}^p \widetilde{g}_{i_n,0}(T_n)\\
    &\quad+\sum_{1\leq n<m\leq p}\int_0^t e_{i_n}^\top E_B(T_n-r)A_0E_B(T_m-r)^\top e_{i_m}\widetilde{\mathfrak{m}}^{(p-2)}(r,(T_l)_{l\neq m,n};w_{-n,-m})dr\\
    &\quad+\sum_{1\leq n<m\leq p}\sum_{j=1}^d\int_0^t e_{i_n}^\top E_B(T_n-r)A_jE_B(T_m-r)^\top e_{i_m} \widetilde{\mathfrak{m}}^{(p-1)}(r,r,(T_l)_{l\neq m,n};w_{-n,-m}^j)dr\\
    &\quad+\sum_{1\leq n<m\leq p}\sum_{j,k=1}^d \int_0^t e_{i_n}^\top E_B(T_n-r)A_{jk}E_B(T_m-r)^\top e_{i_m} \widetilde{\mathfrak{m}}^{(p)}(r,r,r,(T_l)_{l\neq m,n};w_{-n,-m}^{j, k})dr.
\end{split}
\end{equation}

The following two remarks present two important consequences of \eqref{eq:momentfmlavarconstant}.

\begin{remark}
This formulation of the moment formula yields more explicit expressions for the first and second order moments. Indeed, the definition of $\widetilde{g}$ in \eqref{eq:adjfwdvarconst} provides directly a formula for the first order moments $\widetilde{\mathfrak{m}}^{(1)}$. Regarding the second order moments, observe that by taking $p=2$ and $t=T_1=T_2$ in \eqref{eq:momentfmlavarconstant}, we obtain a linear system of integral convolution equations for the functions $\mathbf{f}_{\pi(2,w)}(t)=\widetilde{\mathfrak{m}}^{(2)}(t,t,t;w)$, $w=(i_n)_{n=1}^2\in\{1,\ldots,d\}^2$. More precisely,
\begin{equation*}
\begin{split}
    &\mathbf{f}_{\pi(2,w)}(t)=\prod_{n=1}^2 \widetilde{g}_{i_n,0}(t)+\int_0^t e_{i_1}^\top E_B(t-r)A_0E_B(t-r)^\top e_{i_2}dr\\
    &\quad+\sum_{j=1}^d\int_0^t e_{i_1}^\top E_B(t-r)A_jE_B(t-r)^\top e_{i_2} \widetilde{\mathfrak{m}}^{(1)}(r,r;w_{-n,-m}^j)dr\\
    &\quad+\sum_{j,k=1}^d \int_0^t e_{i_1}^\top E_B(t-r)A_{jk}E_B(t-r)^\top e_{i_2} \mathbf{f}_{\pi(2,w_{-n,-m}^{j, k})}(r)dr.
\end{split}
\end{equation*}
This linear system of convolution equations can be solved using the resolvent of the associated (matrix) kernel. In addition, thanks to \eqref{eq:momentformula}, all the second order moments $\widetilde{\mathfrak{m}}^{(2)}$ can be expressed in terms of $\widetilde{\mathfrak{m}}^{(1)}$ and the functions $\mathbf{f}_{\pi(2,w)}$, $w\in\{1,\ldots,d\}^2$.
\end{remark}

\begin{remark}
If $X$ is an affine Volterra process, i.e. $A_{jk}=0$, then \eqref{eq:momentfmlavarconstant} provides a recursive algorithm to find the moments $\widetilde{\mathfrak{m}}^{(p)}$ of any order $p\in\mathbb{N}$. 
\end{remark}

Remark \ref{rem:linkclassicalmomentfmla} explained how the moment formula in Theorem \ref{th:momentformula} extends the moment formula from the classical to the Volterra framework. The next subsection elucidates that there is one important structural property that is common to the classical and Volterra settings.

\subsection{Moments of polynomial Volterra processes are polynomials}\label{sec:momentsarepoly}

For a polynomial diffusion $X$ starting at $X_0$, i.e. when $K\equiv \mathrm{I}_d$ and $g_0\equiv X_0$ in our framework, \cite[Theorem 3.1]{filipovic2016polynomial} gives the following explicit formula for the moments of $X$ 
\begin{equation}\label{eq:diffusionmoments} 
\mathbb{E}[p(X_t)] = H(X_0)^T\mathrm{e}^{tG}\vec{p}, \quad p\in \mathrm{Pol}_n(\mathbb{R}^d),\,t\geq 0.
\end{equation} 
In \eqref{eq:diffusionmoments}, $H(x)$ is a vector whose components are elements of a basis for $\mathrm{Pol}_n(\mathbb{R}^d)$, $G$ is the matrix of the infinitesimal generator of $X$ restricted to $\mathrm{Pol}_n(\mathbb{R}^d)$, and $\vec{p}$ are the coordinates of the polynomial $p$ with respect to the basis in $H(x)$. As a consequence, the moments of a polynomial diffusion $X$ with initial value $X_0$ are again polynomials in the variable $X_0$. More precisely, 
\begin{equation}\label{eq:uncondmomentsdiffusion}
\mathbb{E}[X_t^{\alpha}] = \sum_{\beta \in \mathbb{N}^d_0, |\beta|\leq |\alpha|}c_{\beta}(t)X_0^{\beta},
\end{equation} 
for some deterministic and time-dependent family of coefficients $\{ c_{\beta}: \beta \in \mathbb{N}^d_0, |\beta|\leq |\alpha|\}$, which can be computed explicitly from \eqref{eq:diffusionmoments}. It turns out that this structural property still holds for the moments $\mathfrak{m}^{(p)}$ in \eqref{eq:momentslift} of a polynomial Volterra process $X$ starting at $X_0$, i.e. when $g_0\equiv X_0$ in \eqref{eq_CPVP}. Furthermore, by an application of Theorem \ref{th:momentformula}, the coefficients can be obtained by solving an integral equation similar to \eqref{eq:momentformula}. Indeed, for $p= |\alpha|=0$ this trivially holds true. For $p=1$ and $w=i_1\in \{1,\dots,d\}$, we have 
\begin{equation}\label{eq:level1function}
\mathfrak{m}^{(1)}(t,T_1;w) = \mathbb{E}[X_{i_i,T_1}]-e_{i_1}^{\top}\int_{t}^{T_1}K(T_1-s)b(\mathbb{E}[X_s])ds, \quad 0 \leq t \leq T_1,
\end{equation} 
where we used \eqref{eq:altfmlalift} and the fact that $b\in \mathrm{Pol}_1(\mathbb{R}^d)$. Plugging the identity \eqref{eq:adjfwdvarconst} into \eqref{eq:level1function} yields 
\begin{align*}
\mathfrak{m}^{(1)}(t,T_1;w) & = e_{i_1}^{\top}\left (\mathrm{I}_d-\int_0^{T_1}R_B(s)ds-\int_{t}^{T_1}K(T_1-s)B\left(\mathrm{I}_d-\int_0^sR_B(u)du\right)ds \right )X_0 \\ & \qquad +e_{i_1}^{\top}\left ( \int_0^{T_1}E_B(s)ds-\int_{t}^{T_1}K(T_1-s)\left (\mathrm{I}_d+B\int_0^sE_B(u)du\right)ds \right )b_0.
\end{align*} 
This readily shows the representation of the form \eqref{eq:uncondmomentsdiffusion} for first-order moments. In the following theorem we exploit the moment-formula in Theorem \ref{th:momentformula}, to generalize this result for higher order moments. 

Before stating the theorem, we recall the notation at the beginning of Subsection \ref{sec:mainformula} and observe that for $p\in\mathbb{N}$ and $w=(i_n)_{n=1}^p\in\{1,\ldots,d\}^p$
\begin{equation}\label{eq:initialcondition}
\mathfrak{m}^{(p)}(0,T_1,\dots,T_p;w) = X_0^{\alpha(w)},
\end{equation} 
where $\alpha(w)$ is the multi-index given by $\alpha_k(w) = \#\{n: i_n=k\}$, $k=1,\ldots,d$. 
\begin{theorem}[Moments are polynomials]\label{thm:momentsarepolynomials}
Suppose that $X$ is a continuous solution to \eqref{eq_CPVP} with $g_0\equiv X_0$ and fix $T\geq 0$. Then,
for all $p\in\mathbb{N}_0$ and  all $w\in\{1,\ldots,d\}^p$, we can express the functions $\mathfrak{m}^{(p)}$ in \eqref{eq:momentslift} as 
\begin{equation} \label{eq:generalcoeffrep}
\mathfrak{m}^{(p)}(t,T_1,\dots,T_p;w) = \sum_{\beta \in \mathbb{N}_0^d, |\beta|\leq p}C^{(p)}_{\beta}(t,T_1,\dots,T_p;w)X_0^{\beta}, \quad (t,T_1,\dots,T_p) \in \mathcal{D}_T^{(p)}, \end{equation}
where the functions $C^{(p)}_{\beta}(\cdot,\cdots,\cdot;w)\in C(\mathcal{D}_T^{(p)},\mathbb{R})$ are independent of $X_0$. \\ In particular, for any $\alpha \in \mathbb{N}_0^d$, there exists a family $\{c_{\beta}\}_{\beta \in \mathbb{N}_0^{d},|\beta| \leq |\alpha|}$ of real-valued continuous functions on $[0,T]$, independent of $X_0$, such that 
\begin{equation}\label{eq:uncondmomentsVolterra}
\mathbb{E}[X_t^{\alpha}] = \sum_{\beta \in \mathbb{N}_0^d, |\beta|\leq |\alpha|}c_{\beta}(t)X_0^{\beta},\quad t\geq 0.
\end{equation} \\
Furthermore -- for any $p\in\mathbb{N}$, $w=(i_n)_{n=1}^p\in\{1,\ldots,d\}^p$, and $\beta \in \mathbb{N}_0^d$ with $|\beta|\leq p$ -- we have 
\begin{equation}\label{eq:coefficientformulas}
\begin{aligned}
    &C_{\beta}^{(p)}(t,T_1,\ldots,T_p;w)=\mathbf{f}^{(p)}_{\beta}(t,T_1,\ldots,T_p;w)\\
    &\quad+\sum_{n=1}^p\sum_{j=1}^d \int_0^t e_{i_n}^\top K(T_n-r)b_j C_{\beta}^{(p)}(r,r,(T_m)_{m\neq n};w_{-n}^j)dr\\
    &\quad +\sum_{1\leq n<m\leq p}\sum_{j,k=1}^d \int_0^t e_{i_n}^\top K(T_n-r)A_{jk}K(T_m-r)^\top e_{i_m}C_{\beta}^{(p)}(r,r,r,(T_l)_{l\neq m,n};w_{-n,-m}^{j, k})dr.
\end{aligned} 
\end{equation}
where
\begin{equation}\label{eq:coefficientformulas_initialcond}
\begin{aligned}
    &\mathbf{f}^{(p)}_{\beta}(t,T_1,\ldots,T_p;w)=1_{\{\beta = \alpha(w)\}}+\sum_{n=1}^p \int_0^t e_{i_n}^\top K(T_n-r)b_0 C_{\beta}^{(p-1)}(r,(T_m)_{m\neq n};w_{-n})dr1_{\{|\beta|<p\}}\\
    &\,+\sum_{1\leq n<m\leq p}\int_0^t e_{i_n}^\top K(T_n-r)A_0K(T_m-r)^\top e_{i_m}C_{\beta}^{(p-2)}(r,(T_l)_{l\neq m,n};w_{-n,-m})dr1_{\{|\beta|<p-1\}}\\
    &\,+\sum_{1\leq n<m\leq p}\sum_{j=1}^d\int_0^t e_{i_n}^\top K(T_n-r)A_jK(T_m-r)^\top e_{i_m}C_{\beta}^{(p-1)}(r,r,(T_l)_{l\neq m,n};w_{-n,-m}^ j)dr1_{\{|\beta|<p\}}.
\end{aligned} 
\end{equation}
\end{theorem} 

\begin{remark}
Referring to the terminology used in \cite{cuchiero2021infinite}, \eqref{eq:coefficientformulas} provides a unique solution to the dual moment formula and
$(C_{\beta}^{(p)}(t,T_1,\ldots,T_p;w))_{t \in [0,T]}$ can be interpreted as deterministic dual process.
\end{remark}

\begin{proof}
Notice that, thanks to \eqref{eq:moments_f_vs_X}, \eqref{eq:uncondmomentsVolterra} directly follows from \eqref{eq:generalcoeffrep} after taking $t=T_1=\cdots=T_p$. The arguments before the statement of the theorem show that \eqref{eq:generalcoeffrep} 
holds  for $p\in \{0,1\}$. Reasoning by induction, assume that for each $1\leq q<p$ we have constructed continuous functions $C^{(q)}_{\beta}(\cdot,\ldots,\cdot;w)$ on $\mathcal{D}_T^{(q)}$ such that \eqref{eq:generalcoeffrep} holds with $p$ replaced by $q$. Corollary \ref{cor:eqncoeffs} in Appendix \ref{app:inteq} shows that, for any $\beta \in \mathbb{N}_0^d$ such that $|\beta|\leq p$, the system of equations \eqref{eq:coefficientformulas} -- seen as a system indexed over the elements $w\in\{1,\ldots,d\}^p$ -- with initial condition \eqref{eq:coefficientformulas_initialcond}, has a unique solution such that $C_{\beta}^{(p)}(\cdot,\ldots,\cdot;w)\in C(\mathcal{D}_T^{(p)})$.

It is straightforward to check that the family of functions on $\mathcal{D}_T^{(p)}$ defined as
\[
\mathfrak{n}^{(q)}(t,T_1,\dots,T_p;w)= \sum_{\beta \in \mathbb{N}_0^d, |\beta| \leq q}C^{(q)}_{\beta}(t,T_1,\dots,T_p;w)X_0^{\beta},\quad (q,w)\in\mathcal{I}^{(p)},
\]
are a continuous solution (in $\mathcal{X}_T^{(p)}$) to the moment-equation \eqref{eq:momentformula}, with $N$ replaced by $p$ and with initial conditions $X_0^{\alpha(w)}$ for each $(q,w)\in\mathcal{I}^{(p)}$. Notice that the initial conditions are the same for the moments $\mathfrak{m}^{(q)}$ by \eqref{eq:initialcondition}. Since by Theorem \ref{th:momentformula} the solution to this moment equation is unique, we conclude that $\mathbf{m}=\mathbf{n}$ and \eqref{eq:generalcoeffrep} holds.
\end{proof} 
\begin{remark}
Observe that the integral equations verified by the coefficients $C_{\beta}$ in \eqref{eq:coefficientformulas} resemble the equations satisfied by the moments in \eqref{eq:momentformula}. This structural property is well-known for classical polynomial processes. Indeed, in the classical case the coefficients of the moments in \eqref{eq:diffusionmoments} are given in terms of the exponential matrix $\mathrm{e}^{tG}$ which solves a linear ODE as well. We could have vectorized the equations for the coefficients \eqref{eq:coefficientformulas} as we did it for the moments in \eqref{eq:momentformula}. We opted, however, to not use the vectorization at this point for clarity of exposition.
\end{remark}

\subsection{Moments of the finite dimensional distributions and uniqueness in law}\label{sec:uniqueness}
In this section we explain how the previous ideas can be extended to characterize moments of the finite dimensional distributions of a weak solution to \eqref{eq_CPVP}. These considerations will allow us to prove a result concerning the uniqueness in law for solutions to \eqref{eq_CPVP}.

Fix $T\ge 0$, $0\leq t_1\leq T$ and $\alpha^1$ a multi-index. Define
\[
 \mathcal{D}_{t_1,T}^{(N)} = \{(t_2,T_1,\ldots,T_N)\in [t_1,T]^{N+1}:t_2\leq \min\{T_1,\ldots,T_N\}\}.
\]
Consider the $\mathbb{R}^{D_N}$-valued function $\mathbf{f}$ on $\mathcal{D}_{t_1,T}^{(N)}$ given by $\mathbf{f}_{\pi(0,\emptyset)}\equiv 1$ and, for $1\leq p\leq N$ and $w=(i_n)_{n=1}^{p}$,
\begin{equation*}    
    \mathbf{f}_{\pi(p,w)}(t_2,T_1,\ldots,T_N)=\mathbf{f}_{\pi(p,w)}(t_2,T_1,\ldots,T_p) = \mathbb{E}[X_{t_1}^{\alpha^1}g_{i_1,t_2}(T_1)\cdots g_{i_p,t_2}(T_p)].
\end{equation*}
Then - by the same considerations as in the proof of Theorem \ref{th:momentformula} -- $\mathbf{f}$ is a bounded function on $\mathcal{D}_{t_1,T}^{(N)}$ satisfying the integral equation 
\begin{align*}
    &\mathbf{f}_{\pi(p,w)}(t_2,T_1,\ldots,T_p) = \mathbb{E}[X_{t_1}^{\alpha^1}g_{i_1,t_1}(T_1)\cdots g_{i_p,t_1}(T_p)]\\
    &\quad+ \sum_{n=1}^p \int_{t_1}^{t_2} e_{i_n}^\top K(T_n-r)b_0 \mathbf{f}_{\pi(p-1,w_{-n})}(r,(T_m)_{m\neq n})dr\\
    &\quad+\sum_{n=1}^p\sum_{j=1}^d  \int_{t_1}^{t_2} e_{i_n}^\top K(T_n-r)b_j \mathbf{f}_{\pi(p,w_{-n}^j)}(r,r,(T_m)_{m\neq n})dr\\
    &\quad+\sum_{1\leq n<m\leq p}\int_{t_1}^{t_2}  e_{i_n}^\top K(T_n-r)A_0K(T_m-r)^\top e_{i_m}\mathbf{f}_{\pi(p-2,w_{-n,-m})}(r,(T_l)_{l\neq m,n})dr\\
    &\quad+\sum_{1\leq n<m\leq p}\sum_{j=1}^d \int_{t_1}^{t_2} e_{i_n}^\top K(T_n-r)A_jK(T_m-r)^\top e_{i_m}\mathbf{f}_{\pi(p-1,w_{-n,-m}^j)}(r,r,(T_l)_{l\neq m,n})dr\\
    &\quad+\sum_{1\leq n<m\leq p}\sum_{j,k=1}^d \int_{t_1}^{t_2} e_{i_n}^\top K(T_n-r)A_{jk}K(T_m-r)^\top e_{i_m}\mathbf{f}_{\pi(p,w_{-n,-m}^{j,k})}(r,r,r,(T_l)_{l\neq m,n})dr
\end{align*}
where $1\leq p\leq N$ and $w=(i_n)_{n=1}^{p}$. The same arguments as in Appendix \ref{app:inteq} show that, for given inputs $(K,b,a)$, this equation has a unique bounded solution with initial condition $\mathbb{E}[X_{t_1}^{\alpha^1}g_{i_1,t_1}(T_1)\cdots g_{i_p,t_1}(T_p)]$. Since we already know that this initial condition only depends on $(K,b,a)$, i.e. it is the same for any weak solution $X$ to \eqref{eq_CPVP}, we conclude that $\mathbf{f}$ only depends on $(K,b,a)$. In particular, the mixed moments of the form $\mathbb{E}[X_{t_1}^{\alpha^1}X_{t_2}^{\alpha^2}]$ are the same for any weak solution $X$ to \eqref{eq_CPVP}. 
\begin{remark}
    It is not surprising that the function $\mathbf{f}$ defined above solves an integral equation similar to \eqref{eq:momentformula}, with initial condition expressed in terms of simple moments at time $t_1$. Indeed, consider the classical case where $K\equiv \mathrm{I}_d$. Then, the tower property of the conditional expectation and the classical conditional moment formula yield
    \begin{equation*}
        \mathbb{E}[X_{t_1}^{\alpha^1}X_{t_2}^{\alpha^2}] = \mathbb{E}[X_{t_1}^{\alpha^1}H(X_{t_1})]\mathrm{e}^{(t_2-t_1)G}\vec p
    \end{equation*}
    where we use the same notation as in \eqref{eq:diffusionmoments} with $n\geq |\alpha^2|$, and where $\vec p$ are the coordinates of the monomial $x^{\alpha^2}$ with respect to the chosen basis. Observe that $\mathbb{E}[X_{t_1}^{\alpha^1}H(X_{t_1})]\mathrm{e}^{(t_2-t_1)G}$ solves a linear ODE with initial condition at $t_1$ given by $\mathbb{E}[X_{t_1}^{\alpha^1}H(X_{t_1})]$.
\end{remark}
A recursive argument, following the same lines as the above mentioned considerations, proves that all weak solutions $X$ of \eqref{eq_CPVP} have the same moments of the form
\begin{equation*}
    \mathbb{E}[X^{\alpha^1}_{t_1}\cdots X_{t_l}^{\alpha^l}]
\end{equation*}
for $l\in\mathbb{N}$, times $t_1\leq\cdots\leq t_l$ and multi-indices $\alpha^1,\ldots,\alpha^l$. Using this observation and arguing as in the proof of \cite[Lemma 4.1]{filipovic2016polynomial}, we deduce the following result.

\begin{theorem}[Uniqueness in law for polynomial Volterra processes]
\label{thm:uniqueness}
    Suppose that for any weak solution $X$ to \eqref{eq_CPVP} and for any $t\ge 0$, the law of $X_t$ is determined by its moments. Then uniqueness in law holds for \eqref{eq_CPVP}.
\end{theorem}

\begin{remark}
    As explained in the proof of \cite[Lemma 4.1]{filipovic2016polynomial}, the hypothesis of Theorem \ref{thm:uniqueness} holds for instance if for any weak solution $X$ to \eqref{eq_CPVP} and $t\geq 0$, there is $\epsilon>0$ such that $\mathbb{E}[\exp(\epsilon|X_t|)]<\infty$. For example, this is the case if any such solution $X$ is bounded.
\end{remark}

\begin{remark}
    Theorem \ref{thm:uniqueness} constitutes a new result regarding uniqueness in law for solutions to stochastic Volterra equations. For affine Volterra processes \cite{abi2019affine} provide uniqueness in law via the Fourier-Laplace transform. Pathwise uniqueness for stochastic Volterra equations has been established for Lipschitz coefficients e.g., in \cite{berger1980volterra} and, for certain coefficients in the one-dimensional case, when the kernel $K$ is regular \cite{abi2019multifactor,promel2023stochastic}  and when the kernel is singular  in \cite{hamaguchi2023weak,wang2008existence}.
\end{remark}

\section{Jump representation of polynomial Volterra processes}
\label{sec:jumprep}

Let $X$ be the Volterra process introduced in \eqref{eq_CPVP} whose dynamics are given by
$$
    X_t = \lambda_0(t) + \int_0^t K(t-s)b(X_s)ds + \int_0^t K(t-s)\sigma(X_s) dW_s,\quad t \in  [0,T].
$$
In this section we consider a different lift, denoted by $\lambda$ that will provide an alternative representation of the moments of $X$, namely in terms of a dual process which is a pure jump process with killing. For this reason we also call the initial condition $\lambda_0$ instead of $g_0$.

\subsection{One-dimensional case}
For the reader's convenience we consider first the one dimensional case setting $d=1$.
Let $(\lambda_t)_{t\in[0,T]}$ be a solution (in a sense made precise in Proposition~\ref{prop:existence} below) of 
\begin{equation}\label{eqn1}
    \lambda_t(x)=\lambda_0(x)+\int_0^t \partial_x \lambda_u(x) du+\int_0^t K(x)\Big[b(\lambda_u(0))du+\sigma(\lambda_u(0)) dW_u\Big],\quad \lambda_0\equiv X_0 \in \mathbb{R},
\end{equation}
 for $x \in \mathbb{R}_+$. Similarly as in \cite[Section 5]{cuchiero2020generalized} and as argued in Remark \ref{rem:connection} below it holds that $\lambda_t(0)=X_t$ and that $\lambda$ actually corresponds to the \emph{Musiela parameterization} of the processes $(g_t(T))_{t \in [0,T]}$ considered in Section \ref{sec:Moments}.
 
As a first step we introduce an appropriate space for the function-valued process $(\lambda_t)_{t\in[0,T]}$.
As in \cite{BK:14, BK:15}, we let $\alpha:\R_+\to[1,\infty)$ be a nondecreasing $C^1$-function such that $\alpha^{-1}\in L^1(\R_+)$. The so-called Filipovi\'c space is then defined by
$$B=\{y\in AC(\R_+,\R) : \|y\|_\alpha<\infty\},$$
 where 
 $AC(\R_+,\R)$ denotes the space of absolutely continuous functions from $\R_+$ to $\R$ and 
$$\|y\|_\alpha^2=|y(0)|^2+\int_0^\infty|y'(x)|^2\alpha(x)dx.$$ 
By \cite[Lemma~3.2]{BK:14} we also know that $B\subseteq \R+C_0(\R_+)$, namely the space of bounded continuous functions with continuous continuation to infinity.
Furthermore, let $(S_t)_{t \in [0,T]}$ denote the left-shift semigroup on $B$, i.e., $S_ty = y (t + \cdot)$. Then $(S_t)_{t \in [0,T]}$ is the $C_0$-semigroup generated by the operator $\partial_x$ (see Filipovi\'c \cite[Theorem 5.1.1]{filipovic2001consistency}).

To ease technicalities and the exposition we make the following assumptions throughout this section (unless otherwise stated).

\begin{assumption}\label{ass1}
\begin{enumerate}
\item\label{iti} We assume that the kernel $K$ satisfies $K \in B$.
\item \label{itii}
Both $b$ and $\sigma$ in \eqref{eqn1} are linear maps. We thus set $b(y)=b_1y$ and $\sigma(y)^2=A_{11} y^2$ for constants $b_1, A_{11} \in \mathbb{R}$.
\end{enumerate}
\end{assumption}
With the shift semigroup at hand we can now define the notion of a mild solution to \eqref{eqn1} as in \cite[Section 7.1]{DZ:14}:

A predictable $B$-valued process $(\lambda_t)_{t \in [0,T ]}$, is said to be a \emph{mild solution} to \eqref{eqn1} if for every $t \in [0,T]$
\[
\mathbb{P}\left[\int_0^t \|\lambda_s\|_{\alpha}^2 ds < \infty\right] =1,
\]
and
\begin{align}\label{eq:eqn2}
\lambda_t=S_t\lambda_0+ \int_0^t S_{t-u}K b_1 \lambda_u(0)du +\int_0^t S_{t-u}K \sqrt{A_{11}} \lambda_u(0)dW_u, \quad \mathbb{P}\text{-a.s.}
\end{align}

\begin{proposition}[Existence and moment bounds for solutions of \eqref{eqn1}]\label{prop:existence}
Under Assumption \ref{ass1} there exists a unique mild solution to \eqref{eqn1} which has a continuous modification. Moreover, for every $k \geq 1$,
we have 
\begin{align}\label{eq:moments}
\mathbb{E}[\sup_{t \in [0,T]} \| \lambda_t\|_{\alpha}^k] \leq C_{T,k, \lambda_0},
\end{align}
where $ C_{T,k, \lambda_0}$ is a constant that depends on $T,k$ and $\lambda_0$. This solution is also an (analytically) weak solution in the sense of \cite[page 161]{DZ:14}.
\end{proposition}

\begin{proof}
We apply \cite[Theorem 7.2]{DZ:14} for the Hilbert space $H=B$ and $U_0=\mathbb{R}$. We thus need to verify the conditions of \cite[Hypotheses 7.1]{DZ:14}. As stated above $\partial_x$ generates the strongly continuous right shift semigroup $(S_t)_{t \in [0,T]}$. Moreover, all measurability criteria are satisfied. Furthermore, since the point evaluations at $0$ are bounded  linear functionals  and the fact that $K \in B$ we have
\[
\| Kb_1g(0) - Kb_1h(0) \|_{\alpha}+ \| K \sqrt{A_{11}} g(0) - K\sqrt{A_{11}} h(0)\|_\alpha \leq C \|g-h\|_{\alpha}, \quad \forall g,h \in B, 
\]
and of course also the linear growth condition. This implies the existence and uniqueness of a mild solution \eqref{eq:eqn2}. Concerning the moment estimate we apply \cite[Equation 7.7]{DZ:14} stating that for $p >2$
\[
\mathbb{E}[\sup_{t \in [0,T]} \| \lambda_t\|_{\alpha}^p] \leq \widehat{C}_{T,p} (1 + \mathbb{E}[\|\lambda_0\|_{\alpha}^p]),
\]
By the initial condition $\lambda_0\equiv X_0 \in \mathbb{R}$ we get a constant on the right hand side.
Moreover, since for $1 \leq k \leq 2$ and $p >2$
\[
\| \lambda_t\|_{\alpha}^k \leq 1 +\| \lambda_t\|_{\alpha}^p
\]
we obtain 
\[
\mathbb{E}[\sup_{t \in [0,T]} \| \lambda_t\|_{\alpha}^k]\leq 1+ \mathbb{E}[\sup_{t \in [0,T]} \| \lambda_t\|_{\alpha}^p ]\leq {C}_{T,k,\lambda_0},
\]
and we get the assertion for all $k \geq 1$.
The last assertion concerning the (analytically) weak solution follows from the same arguments as in \cite[Theorem 6.5]{DZ:14} since 
\[
\mathbb{E}\left[\int_0^T \| K\sqrt{A_{11}} \lambda_t(0)\|^2_{\alpha}\right] < \infty.
\]
\end{proof}

\begin{remark}\label{rem:connection}
Since under Assumption \ref{ass1} $(\lambda_t)_{t \in [0,T]}$ is a mild solution to \eqref{eqn1} with values in $B$
we immediately get from \eqref{eq:eqn2} by evaluating at some fixed $x$ that
\begin{align*}
\lambda_t(x)&=\lambda_0(t+x)+ \int_0^t K(t+x-u) b_1 \lambda_u(0) du + \int_0^t K(t+x-u) \sqrt{A_{11}} \lambda_u(0)dW_u.
\end{align*}
From this we see that $\lambda_t(x)=g_t(t+x)$ where the process $g$ was the defined in \eqref{eq:adjforward}, implying that $\lambda$ just corresponds to the Musiela parameterization of the processes $g$. Moreover, $(\lambda_t(0))_{t \in [0,T]}$ is the unique solution to the Volterra equation 
\[
X_t=X_0 + \int_0^t K(t-u)(b_1 X_u du + \sqrt{A_{11}} X_u dW_u),
\]
where uniqueness is a consequence of the linearity and thus Lipschitz property of the coefficients.
\end{remark}

We are now ready to state in Proposition \ref{prop:jump} the expression of the moments of the polynomial Volterra process in terms of a functional of a pure jump process with killing.
In the proof, we use the notation $\lambda^{\otimes k}(x)$ to denote the product $\lambda(x_1)\cdots \lambda(x_k)$ for $x=(x_1, \ldots, x_k) \in \mathbb{R}^k_+$.

\begin{proposition}[Jump representation of one-dimensional polynomial Volterra processes]\label{prop:jump}
Fix $k\geq 1$ and let $(Y_t)_{t\in [0,T]}$ be the $[0,T]^k$-valued process  generated by $\Gcal_k$
with domain $\mathcal{D}(\mathcal{G}_k)= AC([0,T]^k, \mathbb{R})$  given by
$$\Gcal_k f(x)=1^\top \nabla f(x)+\int (f(\xi)-f(x)) \nu(x,d\xi),$$
for
\begin{align*}
    \nu(x,\cdot)&=b_1\sum_{i=1}^kK(x_i)\delta_{(x_1,\ldots,x_{i-1},0,x_{i+1},\ldots,x_k)}\\
    &\qquad+A_{11}\sum_{i=1}^k\sum_{i<j}K(x_i)K(x_j)\delta_{(x_1,\ldots,x_{i-1},0,x_{i+1},\ldots,x_{j-1},0,x_{j+1},\ldots,x_k)},
\end{align*}
and $Y_0=(0,\ldots,0)$.
Suppose that Assumption \ref{ass1} is in force and assume that the mild solution $(\lambda_t)_{t \in [0,T]}$ to \eqref{eqn1} additionally satisfies
    \begin{equation}\label{eqnbound}
        \E[\sup_{x\in [0,T]^k}|\lambda_t'(x_1)\lambda_t(x_2)\cdots \lambda_t(x_k)|]<\infty.
    \end{equation}
    Then for each $t\in [0,T]$  it holds 
    $$\E[X_t^k]
=
X_0^k\E\left[\exp\left(\int_0^t \kappa(Y_\tau)d\tau\right) \right],$$
where 
$$\kappa(x)=b_1\sum_{i=1}^kK(x_i)+A_{11}\sum_{i=1}^k\sum_{j< i}K(x_i)K(x_j).$$
\end{proposition}

\begin{proof}
Set 
    $$\Mcal_k f(x)=\Gcal_k f(x)+\kappa(x) f(x),$$
for each $f\in \mathcal{D}(\Gcal_k)$. 
Observe that by Itô's formula
and  \eqref{eqn1} (since $(\lambda_t)_{t \in [0,T]}$ is also an analytically weak solution), 
we have 
$$
    \lambda_t^{\otimes k}(x)=X_0^k+\int_0^t \Mcal_k \lambda_u^{\otimes k}(x)du
    +\int_0^t \mathcal Q \lambda_u^{\otimes k}(x)dW_u
$$
for
$$\mathcal Q f(x)=\sqrt{A_{11}} \sum_{i=1}^kK(x_i) f(x_1,\ldots,x_{i-1},0,x_{i+1},\ldots,x_k).$$
Since $$\E[\sup_{t\in [0,T]}|\lambda_t(x_1)\lambda_t(x_2)\cdots \lambda_t(x_k)|]<\infty$$
due to \eqref{eq:moments} and as $K \in B$
the third term is a  true martingale and we thus have
$$\E[\lambda_t^{\otimes k}(x)]=
X_0^k+\int_0^t \E[\Mcal_k \lambda_u^{\otimes k}(x)]du.
$$
Set then $m_t(x)=\E[\lambda_t^{\otimes k}(x)]$ for each $x\in [0,T]^k$. Since \eqref{eqnbound} holds,  Leibniz rule and Fubini yield
\begin{align*}
    \Mcal_km_t(x)
    &=
1^\top \nabla \E[\lambda_t^{\otimes k}(\cdot)](x)+\int \E[\lambda_t^{\otimes k}(\xi)] \nu(x,d\xi)\\
&= \E[1^\top \nabla \lambda_t^{\otimes k}(x)]+\E\left[\int \lambda_t^{\otimes k}(\xi) \nu(x,d\xi)\right],
\end{align*}
proving that $\Mcal_km_t(x)= \E[\Mcal_k \lambda_t^{\otimes k}(x)]$ and thus that
$$m_t(x)=
X_0^k+\int_0^t \Mcal_km_u(x)du.
$$
Note that since $m_s\in \Dcal(\Gcal_k)$ for each $s\in [0,T]$, by Itô's formula we get 
$$dm_s(Y_t)=\Gcal_k m_s(Y_t) dt+
dM_t(m_s),
$$
for
$$M_t(f)=\int_0^t \int \big(f(\xi)-f(Y_t)\big) (\mu_t(dt,d\xi)-\nu(Y_t,d\xi)dt),$$ where $\mu_t$ is the jump measure corresponding to $Y$. Fixing now $t\in [0,T]$ we thus get that
\begin{align*}
    d m_{t-s}(Y_s)
    &=
-\partial_s m_{t-s}(Y_s) ds
+\Gcal_k m_{t-s}(Y_s) ds
+dM_s(m_{t-s})\\
&=
-\Mcal_k m_{t-s}(Y_s) ds
+\Gcal_k m_{t-s}(Y_s) ds
+dM_s(m_{t-s}),
\end{align*}
and similarly for $Z_s=\exp(\int_0^s \kappa(Y_\tau)d\tau) m_{t-s}(Y_s)$ we get
\begin{align*}
    dZ_s&=-\exp\left(\int_0^s \kappa(Y_\tau)d\tau\right) \Mcal_km_{t-s}(Y_s)ds\\
    &\qquad+\exp\left(\int_0^s \kappa(Y_\tau)d\tau\right) \Gcal_k m_{t-s}(Y_s)ds
    +\exp\left(\int_0^s \kappa(Y_\tau)d\tau\right)dM_s(m_{t-s})\\
    &\qquad+
    \kappa(Y_s)\exp\left(\int_0^s \kappa(Y_\tau)d\tau\right)m_{t-s}(Y_s)ds\\
    &=\exp\left(\int_0^s \kappa(Y_\tau)d\tau\right)dM_s(m_{t-s}),
\end{align*}
showing that $(Z_s)_{s\in[0,t]}$ is a local martingale.
Since $\sup_{t,x\in [0,T]}|m_t(x)|<\infty$ and $K$ is bounded we can conclude  that
$$\E\left[\exp\left(\int_0^t \kappa(Y_\tau)d\tau\right)X_0^k\right]
=\E[Z_t]
=\E[Z_0]
=m_t(0,\ldots,0)
=\E[\lambda_t^{\otimes k}(0,\ldots,0)]
=\E[X_t^k].$$
\end{proof}

\begin{example}
    Note that for $K\equiv \lambda_0\equiv 1$ we get that $X=X_0\mathcal E((b_1t+\sqrt{A_{11}} W_t)_{t\in [0,T]})$. The corresponding representation yields
    $$\E[X_t^k]
=
X_0^k\exp\Big(\int_0^t b_1k+A_{11} k(k-1)d\tau\Big)
=X_0^k\exp((b_1k+\frac 1 2 A_{11}k(k-1))t),$$
as expected.
\end{example}

\begin{remark}
Observe that Proposition \ref{prop:jump} shows that in the current homogeneous case where both $b$ and $\sigma$ are linear functions, the $k$th moment of the Volterra process is just a monomial in $X_0^k$ with coefficient
$\E[\exp(\int_0^t \kappa(Y_\tau)d\tau)]$. This is thus a special case of Theorem \ref{thm:momentsarepolynomials}.
Note that the proof of Proposition \ref{prop:jump} also shows that
\[
\mathbb{E}[\lambda_t(x_1)\cdots \lambda_t(x_k)]= \E[\lambda_t^{\otimes k}(x)]=\E_{Y_0=x} \left[\exp\left(\int_0^t \kappa(Y_\tau)d\tau\right)\lambda_0(Y_{1,t})\cdots\lambda_0(Y_{k,t})\right],
\]
so that we also get a jump representation for $\mathbf{m}(t, t+x_1, \ldots, t+x_k)$ (where we apply the notation of Section \ref{sec:mainformula}). If $\lambda_0 \equiv X_0$, we obtain again a representation that involves only $X_0^k$ and thus again a special case of Theorem~\ref{thm:momentsarepolynomials}.

Let us also explicitly draw the connection to the dual process approach outlined in the introduction via \eqref{eq:dualgen}.
The infinitesimal generator $\mathcal{A}$ of the process $\lambda$ applied to cylindrical functions $f(\lambda)=g(\langle a_1, \lambda\rangle, \ldots, \langle a_k, \lambda\rangle)$ where $g \in C^2(\mathbb{R}^k)$ and $a_1, \ldots, a_k$ are continuous linear functionals is given by
\begin{align*}
\mathcal{A}f(\lambda)&= \sum_{i=1}^k \partial_i g (\langle a_1, \lambda\rangle, \ldots, \langle a_k, \lambda\rangle) \langle\partial_x \lambda + K b_1 \lambda(0), a_i \rangle \\
& \quad + \frac{1}{2} \sum_{i,j=1}^k \partial_{ij} g(\langle a_1, \lambda\rangle, \ldots, \langle a_k, \lambda\rangle) \langle K \sqrt{A_{11}} \lambda(0), a_i \rangle \langle K \sqrt{A_{11}} \lambda(0), a_j \rangle.
\end{align*}
Letting $g(y)=\prod_{i=1}^k y_i$ and $\langle a_i, \lambda \rangle= \lambda(x_i)$, i.e.~the point evaluations for some fixed $x=(x_1, \ldots, x_k)$, we thus obtain
\begin{align*}
&\mathcal{A}\lambda(x_1)\cdots \lambda(x_k)\\
&\quad =
\sum_{i=1}^k  \lambda(x_1)\cdots \lambda(x_{i-1})\lambda(x_{i+1}) \cdots \lambda(x_k)(\partial_x \lambda(x_i) + K(x_i) b_1 \lambda(0) )\\
& \quad \quad + \frac{1}{2} \sum_{i,j=1}^k  \lambda(x_1)\cdots \lambda(x_{i-1})\lambda(x_{i+1}) \cdots  \lambda(x_{j-1})\lambda(x_{j+1}) \cdots\lambda(x_k)
K(x_i)  K(x_j) A_{11} \lambda^2(0).
\end{align*}
This is exactly the same as $\mathcal{L}$ in \eqref{eq:dualgen} with $\mathcal{L}\lambda(x_1)\cdots \lambda(x_k)=\mathcal{G}_k \lambda(x_1)\cdots \lambda(x_k) + \kappa(x) \lambda(x_1)\cdots \lambda(x_k)$
 and the process $U$ from the introduction corresponds to the process $Y$ killed at rate $\kappa$.
\end{remark}

\begin{remark}\label{rem:nonhomogen}
We illustrate now how the proposed method works without supposing the homogeneity condition of Assumption \ref{ass1}\ref{itii}. Consider coefficients $b_0,b_1,A_0,A_1,A_{11}$ such that 
    $b(x)=b_0+b_1x$
    and
    $\sigma(x)^2=A_0+A_1x+A_{11}x^2.$
    Fix $k\geq 1$ and let $(Y_t)_{t\in[0,T]}$ be the  $([0,T]\cup{\{\dag\}})^k$-valued process   generated by the linear operator $\Gcal_k$
with domain 
\begin{align*}
    \overline{\mathcal{D}(\mathcal{G}_k)}=&\{f:([0,T]\cup\{\dag\})^k\to\R \colon f\text{ is symmetric,}\\
    &\qquad f(\underbrace{\cdot,\ldots,\cdot}_{i},\underbrace{\dag,\ldots,\dag}_{k-i})|_{[0,T]^i}\in AC([0,T]^i, \mathbb{R}),\, i\in\{1,\ldots,k\},\,
    f(\dag,\ldots,\dag)=1 \}
\end{align*}
and given by 
$$\Gcal_k f(y)=1^\top \nabla f(y)+\int (f(\xi)-f(y)) \nu(y,d\xi),$$
for $\nabla_{x_i}f(\dag)=0$, $K(\dag)=0$ ,
\begin{align*}
\nu(y,\cdot)&=b_0\sum_{i=1}^kK(y_i)\delta_{(y_1,\ldots,y_{i-1},\dag,y_{i+1},\ldots,y_k)}+b_1\sum_{i=1}^kK(y_i)\delta_{(y_1,\ldots,y_{i-1},0,y_{i+1},\ldots,y_k)}\\
&\qquad+A_0\sum_{i=1}^k\sum_{j< i}K(y_i)K(y_j)\delta_{(y_1,\ldots,y_{i-1},\dag,y_{i+1},\ldots,y_{j-1},\dag,y_{j+1},\ldots,y_k)}\\
&\qquad+\frac 1 2A_1\sum_{i=1}^k\sum_{j\neq  i}K(y_i)K(y_j)\delta_{(y_1,\ldots,y_{i-1},0,y_{i+1},\ldots,y_{j-1},\dag,y_{j+1},\ldots,y_k)}\\
&\qquad+A_{11}\sum_{i=1}^k\sum_{j< i}K(y_i)K(y_j)\delta_{(y_1,\ldots,y_{i-1},0,y_{i+1},\ldots,y_{j-1},0,y_{j+1},\ldots,y_k)},
\end{align*}
 and $Y_0=(0,\ldots,0)$.
Suppose that  Assumption \ref{ass1}\ref{iti} holds true and assume that the mild solution $(\lambda_t)_{t \in [0,T]}$ to \eqref{eqn1} additionally satisfies
    $$\E[\sup_{y\in [0,T]^k}|\lambda_t'(y_1)\lambda_t(y_2)\cdots \lambda_t(y_k)|]<\infty.$$ 
    Then for each $t\in [0,T]$ it holds 
    $$\E[X_t^k]
=
\E\left[\exp\left(\int_0^t \kappa(Y_\tau)d\tau\right)X_0^{\sum_{i=1}^k1_{\{Y_{i,t}\neq \dag\}}} \right],$$
where 
$$\kappa(y)=b(1)\sum_{i=1}^kK(y_i)+\sigma(1)^2\sum_{i=1}^k\sum_{j< i}K(y_i)K(y_j).$$

Observe  that there is a direct connection between this representation and Remark~\ref{rmk:musiela}. Indeed, given a solution $(f_t)_{t\in[0,T]}$ of
$$\partial_tf_t(y)=\Gcal_kf_t(y)$$
we obtain that $ \overline{\mathbf{m}}_{\pi(p,w)}(t,x_1,\ldots,x_p)=f_t(x_1,\ldots,x_p,\dag,\ldots,\dag)$, $p=0,\ldots, k$ solves \eqref{eq:momentPDE}. Moreover, noting that
$$\E[X_t^k]=\sum_{j=0}^k \E\Big[\exp\left(\int_0^t \kappa(Y_\tau)d\tau\right)\Big|\sum_{i=1}^k1_{\{Y_{i,t}\neq \dag\}}=j\Big]\P\Big(\sum_{i=1}^k1_{\{Y_{i,t}\neq \dag\}}=j\Big)X_0^j,$$
we can also establish a direct connection with Theorem~\ref{thm:momentsarepolynomials} and in particular equation \eqref{eq:uncondmomentsVolterra}.
\end{remark}

\subsection{Multivariate case}
We move now to the $d$-dimensional setting, letting again $X$ be the process defined in \eqref{eq_CPVP}.
%
%
In this case we consider a solution $(\lambda_t)_{t\in [0,T]}$ of the SPDE given by
\begin{equation}\label{eqn2}
\begin{aligned}
\lambda_t(x, i)&= X_{0,i}+ \int_0^t \frac{d}{dx} \lambda_u(x,i) du +  \sum_{j=1}^d e_i^\top K(x)b_j  \lambda_u(0,j)  du \\& \quad 
+ \sum_{j_1,j_2=1}^d  \sqrt{ \lambda_u(0,j_1)    \lambda_u(0,j_2)  }e_i^\top K(x)\sqrt{A_{j_1j_2}} dW_u,
\end{aligned}
\end{equation}
for $x\in \R_+$ and $i\in \{1,\ldots,d\}$ and where $e_i$ denotes the canonical basis vectors. The Hilbert space for the corresponding function-valued process $(\lambda_t)_{t\in [0,T]}$ is given by
$$B_d=\{f:\mathbb{R}_+ \times \{1, \ldots d\}\to\mathbb{R}\colon f(\cdot,i)\in B\},$$
with
$$\|y\|_{\alpha,d}^2=\sum_{i=1}^d\Big (|y(0,i)|^2+\int_0^\infty|y'(x,i)|^2\alpha(x)dx\Big).$$ 

We first start by establishing the analog of Proposition \ref{prop:existence}. Even though we believe that existence of an analytically weak solution can be proved via the martingale problem approach we here just focus on the case $A_{j_1 j_2}=0$ whenever $j_1 \neq j_2$ so that we only have  linear 
terms in the diffusion part and can apply similar arguments as in Proposition \ref{prop:existence}. 

For completeness we recall the notion of a mild solution in the current setting when $A_{j_1 j_2}=0$ for $j_1 \neq j_2$.
A predictable $B_d$-valued process $(\lambda_t)_{t \in [0,T ]}$, is said to be a \emph{mild solution} to \eqref{eqn2} if for every $t \in [0,T]$
\[
\mathbb{P}\left[\int_0^t \|\lambda_s\|_{\alpha,d}^2 ds < \infty\right] =1,
\]
and for every $i \in \{1, \ldots, d\}$ 
\begin{equation}\label{eq:eqn3}
\begin{aligned}
\lambda_t(\cdot,i)&= \underbrace{S_t\lambda_0(\cdot,i)}_{\equiv X_{0,i}}+  \sum_{j=1}^d S_{t-u}e_i^\top Kb_j  \lambda_u(0,j)  du 
+ \sum_{j=1}^d  S_{t-u}  e_i^\top K\sqrt{A_{jj}} \lambda_u(0,j)  dW_u, \quad \mathbb{P}\text{-a.s.}
\end{aligned}
\end{equation}

  Throughout this section we shall assume the following condition on $K$.

\begin{assumption}\label{ass2}
We assume that $K$ satisfies $K_{ij} \in B$ for all $i,j \in \{1, \ldots, d\} 
$.
\end{assumption}

\begin{proposition}[Existence and moment bounds for solutions of \eqref{eqn2}]\label{prop:existence1}
Let Assumption~\ref{ass2} hold true and 
suppose that $A_{j_1 j_2}=0$ for $j_1 \neq j_2$.
Then there exists a unique mild solution to \eqref{eqn2} which has a continuous modification. Moreover, for every $\p \geq 1$,
we have 
\begin{align}\label{eq:moments2}
\mathbb{E}[\sup_{t \in [0,T]} \| \lambda_t\|_{\alpha,d}^\p ] \leq C_{T,\p , X_0},
\end{align}
where $ C_{T,\p , X_0}$ is a constant that depends on $T, \p $ and $X_0$. This solution is also an (analytically) weak solution in the sense of \cite[page 161]{DZ:14}.
\end{proposition}  

\begin{proof}
We again apply Theorem \cite[Theorem 7.2]{DZ:14} for the Hilbert space $H=B_d$ and $U_0=\mathbb{R}^d$.
Verifying the conditions of \cite[Hypotheses 7.1]{DZ:14} is completely analogous to the one dimensional case and we obtain existence and uniqueness of a mild solution to \eqref{eq:eqn3}. The moments estimates and the assertion concerning the weak solution also follow analogously.
\end{proof}

In the next proposition we establish the jump representation in the multivariate case. Here, we also allow for $A_{j_1 j_2}\neq 0$ when $j_1 \neq j_2$. Note  that all the assumptions in Proposition \ref{prop:jumpmulti} except of \eqref{eq:derivativeass} are satisfied under the conditions of Proposition \ref{prop:existence1}.

\begin{proposition}[Jump representation of $d$-dimensional polynomial Volterra processes]\label{prop:jumpmulti}
Fix $k\geq1$ and let $(Y_t)_{t\in [0,T]}$ be the $([0,T]\times \{1, \ldots, d\})^\p $-valued process   generated by
$\mathcal{G}_\p $ with domain $\mathcal{D}(\mathcal{G}_\p )=AC(([0,T]\times\{1, \ldots,d\})^\p , \mathbb{R})$ given by 
\begin{align*}
\Gcal_\p  f(x,i)=1^\top\nabla_x f(x,i)+ \int (f(  \xi, \ell)- f(x,i)) \nu((x,i), d(\xi, \ell)),
\end{align*}
for
\begin{align*}
\nu((x,i), \cdot)&=\sum_{m=1}^\p  \sum_{j=1}^{d} e_{i_m}^\top K(x_m) b_{j} \delta_{\{((x_1,i_1),\ldots,(0,j), \ldots,(x_\p ,i_\p ))\}}\\
&\quad +  \sum_{m=1}^\p \sum_{n<m} \sum_{j_1,j_2=1}^{d}  e_{i_m}^\top K(x_m) A_{j_1j_2}K(x_n)^{\top}e_{i_n} \delta_{\{((x_1,i_1),\ldots,(0,j_1), \ldots,(0,j_2), \ldots,(x_\p ,i_\p ))\}}
\end{align*}
and 
$$Y_0=(\underbrace{(0,1)\ldots,(0,1)}_{\p _1 \text {-times}},\underbrace{(0,2), \ldots, (0,2)}_{\p _2 \text{-times}}, \ldots, \underbrace{(0,d), \ldots, (0,d)}_{\p _d \text{-times}}).$$
Suppose that Assumption \ref{ass2} holds true and assume that $(\lambda_t)_{t\in[0,T]}$ is a $B_d$-valued (analytically) weak solution to \eqref{eqn2} satisfying 
    \begin{align}
&\E[\sup_{t\in [0,T],i}|\lambda_t(x_1,i_1)\cdots \lambda_t(x_\p ,i_\p )|]<\infty \quad \text{and}\notag\\
   & \E[\sup_{x\in[0,T]^\p ,i}|(\partial_x \lambda_t(x_1,i_1))\lambda_t(x_2,i_2)\cdots \lambda_t(x_\p ,i_\p )|]<\infty.\label{eq:derivativeass}
    \end{align}
    Then for each  $t\in [0,T]$ and each multi-index $\mathbf{\p }=(\p _1, \ldots \p _d)$  with $|\mathbf{k}| = k$, it holds
    \begin{align*}
\E[X_t^{\mathbf{\p }}]&=\E[X_{t,1}^{\p _1}\cdots X_{t,d}^{\p _d}]
=X_0^{\mathbf{\p }}
\E\left[\exp\left(\int_0^t \kappa(Y_\tau)d\tau\right) \right],
\end{align*} where
\begin{align*}
\kappa(x,i)&=\sum_{m=1}^\p  \sum_{j=1}^{d} e_{i_m}^\top K(x_m) b_{j}  +  \sum_{m=1}^\p \sum_{n<m} \sum_{j_1,j_2=1}^{d}  e_{i_m}^\top K(x_m) A_{j_1j_2}K(x_n)^{\top}e_{i_n}.
\end{align*}
\end{proposition}
\begin{proof}
    Set 
    $$\Mcal_\p  f(x,i)=\Gcal_\p  f(x,i)+\kappa(x,i) f(x,i),$$
for each  $f\in \Dcal(\Gcal_\p )$ and
observe that by Itô's formula  it holds 
$$
    \lambda_t^{\otimes \p }(x,i)=X_{0,i_1}\cdots X_{0,i_\p }+\int_0^t \Mcal_\p  \lambda_u^{\otimes \p }(x,i)du
    +\int_0^t \mathcal Q \lambda_u^{\otimes \p }(x,i)dW_u.
$$
for
\begin{align*}
    \mathcal Q f(x,i)
    &=
\sum_{m=1}^\p 
\sum_{j_1,j_2=1}^d  e_{i_m}^\top K(x_m)\sqrt{A_{j_1j_2}}\\
&\qquad\times\sqrt{ f((x_1,i_1),\ldots,(0,j_1),\ldots, (x_d,i_d))   f((x_1,i_1),\ldots,(0,j_2),\ldots,(x_d,i_d))}. 
\end{align*}
Since the third term is a  true martingale, we have
$$\E[\lambda_t^{\otimes \p }(x,i)]=
X_{0,i_1}\cdots X_{0,i_\p }+\int_0^t \E[\Mcal_\p  \lambda_u^{\otimes \p }(x,i)]du.
$$
Set then $m_t(x,i)=\E[\lambda_t^{\otimes \p }(x,i)]$. Since \eqref{eq:derivativeass} holds, Leibniz rule and Fubini yield
\begin{align*}
    \Mcal_\p m_t(x,i)&=
1^\top \nabla_x \E[\lambda_t^{\otimes \p }(\cdot,i)](x)+\int \E[\lambda_t^{\otimes \p }(\xi,\ell)] \nu((x,i),d(\xi,\ell))\\
&= \E[1^\top \nabla_x \lambda_t^{\otimes \p }(x,i)]+\E\left[\int \lambda_t^{\otimes \p }(\xi,\ell) \nu((x,i),d(\xi,\ell))\right],
\end{align*}
proving that $\Mcal_\p m_t(x,i)= \E[\Mcal_\p  \lambda_t^{\otimes \p }(x,i)]$ and thus that
$$m_t(x,i)=
X_{0,i_1}\cdots X_{0,i_\p }+\int_0^t \Mcal_\p m_u(x,i)du.
$$
Since  $m_s$ lies in the domain of $\Gcal_\p $ for each $s\in [0,T]$ proceeding as in the one dimensional case we get that the process $(Z_s)_{s\in [0,t]}$ for 
$$Z_s=\exp\left(\int_0^s \kappa(Y_\tau)d\tau\right) m_{t-s}(Y_s)$$
is a true martingale and thus  that 
$$X_0^{\mathbf{\p }}
\E\left[\exp\left(\int_0^t \kappa(Y_\tau)d\tau\right) \right]
=m_t(Y_0)
=\E[\lambda_t^{\otimes \p }(Y_0)]
=\E[X_t^{\mathbf \p }].$$
\end{proof}

\begin{example}
    Observe that for 
    $K\equiv \mathrm{I}_d$, $\lambda_0\equiv (1,\ldots,1)$, $b_i=\beta_ie_i$, $A_{ij}= \sigma_i^\top\sigma_je_ie_j^\top$ for some $\beta_i\in\R$, $\sigma_i\in \R^d$ we get that
    $$X_{i}=\mathcal E((\beta_it+\sum_{j=1}^d \sigma_{ij}W_t^j)_{t\in[0,T]}).$$
    In this setting 
    \begin{align*}
\nu((x,i), \cdot)&=\sum_{m=1}^\p  \sum_{j=1}^{d} e_{i_m}^\top  \beta_je_j \delta_{\{\ldots,(0,j), \ldots\}}
+ \sum_{m=1}^\p \sum_{n<m} \sum_{j_1,j_2=1}^{d}  \sigma_{j_1}^\top\sigma_{j_2}e_{i_m}^\top e_{j_1}e_{j_2}^\top e_{i_n} \delta_{\{\ldots,(0,j_1), \ldots,(0,j_2), \ldots\}}\\
&=\sum_{m=1}^\p    \beta_{i_m} \delta_{\{\ldots,(0,i_m), \ldots\}}
+ \sum_{m=1}^\p \sum_{n<m} \sigma_{i_m}^\top\sigma_{i_n} \delta_{\{\ldots,(0,i_m), \ldots,(0,i_n), \ldots\}},
\end{align*}
showing that the second component of each $Y_i$ is constant over time. Since
$$\kappa(x,i)=\sum_{m=1}^\p    \beta_{i_m} 
+ \sum_{m=1}^\p \sum_{n<m} \sigma_{i_m}^\top\sigma_{i_n}
=\sum_{m=1}^\p    \beta_{i_m} 
+ \frac 12\sum_{m,n=1}^\p  \sigma_{i_m}^\top\sigma_{i_n}-\sum_{m=1}^\p \sigma_{i_m}^\top\sigma_{i_m},$$
    the corresponding representation yields
        $$\E[X_t^\mathbf{\p }|\lambda_0]
=\exp\Big(\Big(\sum_{j=1}^d\p _j\beta_j+\frac 1 2\sum_{j_1, j_2=1}^d \p _{j_1}\p _{j_2}\sigma_{j_1}^\top\sigma_{j_2}-\sum_{j=1}^d\p _j\sigma_j^\top\sigma_j\Big)t\Big),$$
as expected.
\end{example}

\begin{remark}
Note that in analogy to Remark \ref{rem:nonhomogen} also in the multivariate setting the non-homogeneous case can be treated. This can be done by adding zero indices in the definition of $\kappa$ and $\nu$, i.e.~including $b_0$, and $A_{0i}=A_{i0}=A_i$ for $i=0, 1, \ldots, d$, and changing the corresponding jumps
to $(\dag, 0)$ (instead of $(0,j)$ as it is the case for $j=1, \ldots, d$).
\end{remark}

\section{Volterra processes in the unit ball}
\label{sec:unitball}
In this section, we are interested in constructing  general Volterra processes \eqref{eq_CPVP},
that remain in the unit ball $\mathcal B$ of $\R^d$.
To achieve this, structural conditions on the coefficients $b: \R^d \rightarrow \R^d, \sigma: \R^d \rightarrow \R^{d\times d}$ and the kernel $K \in L_{\text{loc}}^2$ are needed. In a first step, we consider more general continuous coefficients  with linear growth \begin{equation}\label{eq:lineargrowth}
    |b(x)|\lor |\sigma(x)| \leq c(1+|x|), \quad  x \in \R^d.
\end{equation} 
It follows from \cite[Theorem 3.4]{abi2019affine} that under Assumption \ref{ass:KcontinL2} and \eqref{eq:lineargrowth},  there exists a weak solution to \eqref{eq_CPVP} with values in $\mathbb R^d$. For the rest of this section, we will assume that the kernel $K$ is scalar $K:\mathbb{R}_+ \rightarrow \mathbb{R}$.

 The next theorem provides the existence of a $\mathcal B$-valued solution to the stochastic Volterra equation \eqref{eq_CPVP} under structural conditions on the coefficients $(b,\sigma)$ and the kernel $K$.   We denote by $\partial \mathcal B$ the boundary of $\mathcal B$, that is 
$$ \partial \mathcal B   = \{ x \in \R^d: x^\top x = 1  \}.$$

\begin{theorem}[Existence of Volterra processes in the unit ball]\label{T:invarianceball} Fix a scalar kernel $K:[0,T] \to \mathbb R$ that satisfies Assumptions \ref{ass:KcontinL2} and \ref{ass:K orthant}.
    Assume that $b$ and $\sigma$ are continuous, with linear growth \eqref{eq:lineargrowth} and
    \begin{align}\label{eq:ballcoeff}
         x^\top b(x) \leq 0 \text{ and } 
        \sigma(x) = 0, \quad  x \in \partial \mathcal B.  
    \end{align}
    For any $X_0 \in \mathcal B$, there exists a continuous weak solution $X$ to \eqref{eq_CPVP} such that $X_t \in \mathcal B$ a.s. for every $t\in [0,T]$.
\end{theorem}

\begin{proof}
    The proof is given in Section~\ref{S:proofinvariance} below.
\end{proof}

\begin{remark}
Theorem~\ref{T:invarianceball} allows the construction of Volterra processes living in the more general domain  $\{x \in \R^d: x^{\top}Qx \leq 1\}$, for some positive definite $d\times d$ matrix $Q$. Indeed, it suffices to construct a weak solution $X$ for the unit ball  and set $Y=Q^{-1/2}X$. 
\end{remark}

\subsection{Proof of Theorem \ref{T:invarianceball}}\label{S:proofinvariance}
\begin{proof}[Proof of Theorem \ref{T:invarianceball}]For any $n\in \mathbb{N}$ consider the coefficients $b^n$ and $\sigma^n$ by \begin{align*}
    b^n(x) &= \begin{cases}
        b(\frac{x}{1-1/n})-\frac{x}{n(1-1/n)}, & \vert x\vert \leq 1-1/n \\ b(\frac{x}{\vert x \vert})-\frac{x}{n\vert x \vert}, & else
    \end{cases} \\   \sigma^n(x) &= \begin{cases}
        \sigma \left (\frac{x}{1-1/n}\right ), & \vert x\vert \leq 1-1/n  \\ 0, & else.
    \end{cases}
\end{align*} By the assumptions for $b$ and $\sigma$, it follows that $b^n$ and $\sigma^n$ are continuous, and they satisfy the linear growth condition \eqref{eq:lineargrowth} uniformly in $n$, with common constant $\widehat{C}$. Denote by $X^n$ a continuous weak solution to \eqref{eq_CPVP} with respect to $(b^n,\sigma^n)$, obtained from \cite[Theorem 3.4]{abi2019affine}. Applying tightness and stability results for stochastic Volterra equations, see for instance \cite[Lemma A.1 and A.2]{abi2019affine}, we have $X^n \Longrightarrow \hat{X}$, where $\hat{X}$ solves \eqref{eq_CPVP} with coefficients $\hat{b}(x)=1_{\mathcal{B}}(x)b(x) + 1_{\mathcal{B}^c}(x)b(x/\vert x \vert) $ and $\hat{\sigma}(x) = 1_{ \mathcal{B}}(x)\sigma(x)$. If we can prove that $X^n \in \mathcal{B}$ almost surely for all $n\in \mathbb{N}$, then $\hat{X}\in \mathcal{B}$ almost surely and since $\hat{b}|_{\mathcal{B}}=b$ and $\hat{\sigma}|_{\mathcal{B}}=\sigma$, the claim follows.

Fix $n\in \mathbb{N}$ and denote $X=X^n$. Note that the coefficients $(b^n,\sigma^n)$ satisfy the stronger conditions \begin{equation}\label{eq: newcoeffcond}
    \sigma^n(x) = 0 \quad \text{and} \quad x^{\top}b^n(y)\leq 0, \quad \vert y \vert,\vert x\vert > 1-1/n \text{ and } \vert x-y \vert < \epsilon,
\end{equation} for some $\epsilon$ depending on $n$ and the constant from the linear growth condition. Indeed let $y$ such that $\vert y \vert > 1-1/n$ and denote by $\hat{y} = y/\vert y \vert  \in \partial \mathcal{B}$, then \begin{align*}
    x^{\top}b^n(y) = y^{\top}b^n(y)+(x-y)^{\top}b^n(y) &= \vert y \vert (\hat{y}^{\top}b(\hat{y})-(1/n) \hat{y}^{\top}\hat{y})+(x-y)^{\top} b^n(y) \\ & \leq \vert y \vert (\hat{y}^{\top}b(\hat{y})-(1/n) \hat{y}^{\top}\hat{y})+\vert x-y \vert \left (\vert b(\hat{y})\vert+1/n \right)\\ &\leq \vert y \vert (\hat{y}^{\top}b(\hat{y})-(1/n) \hat{y}^{\top}\hat{y})+\vert x-y \vert (2C+1) \\ & \leq -\frac{(n-1)}{ n^2} + (2C+1)\vert x-y\vert \\ & \leq 0,
\end{align*} whenever $\vert x-y \vert \leq \frac{n-1}{(2C+1)n^2}=\epsilon$, where we used Cauchy-Schwarz for the second inequality with $C$ denoting the linear growth constant of $b$, and the assumption \eqref{eq:ballcoeff} for the third inequality. \\ \\ Now consider the stopping time $\tau = \inf\{t \geq 0: X_t \notin  \mathcal{B}\}$, we aim to prove that $\mathbb{P}(\tau < \infty )=0$.  Applying similar reasoning to \cite[Theorem 3.6]{abi2019affine}, under Assumption~\ref{ass:K orthant}, for any stopping time $\tau$ and $h>0$ we have
\begin{equation} \label{eq:orth:0}
\begin{aligned}
X_{\tau+h} 
&= \left( 1-(\Delta_h K * L)(\tau)\right)X_{0} + (\Delta_h K* L)(0)X_{\tau}  + (d(\Delta_h K * L)*(X))({\tau}) \\
&\quad + \int_0^\infty  1_{(\tau,\tau+h]}(s)K(\tau+h-s) b^n(X_s)ds\\
&\quad +\left(\int_0^\infty 1_{(\tau,\tau+h]}(s)K(t+h-s) \sigma^n(X_s)dW_s\right)\Bigg 
 |_{t=\tau}  \\
& =: a_0(h) + a_1(h) + a_2(h).
\end{aligned}
\end{equation} Now for $\epsilon' >0$, we define the events $$\Omega^{\epsilon'} = \left \{\tau<\infty, X^{\top}_sb^n(X_u)\leq 0 \text{ and } \sigma^n(X_u)=0 \quad \forall u,s \in [\tau, \tau+\epsilon'[ \right \}.$$ On the event $ \Omega^{\epsilon'}$, for all $h\in (0,\epsilon')$, we have $a_2(h)=0$, and for $\psi(t)= \vert X_{\tau+t} \vert $
\begin{align}\label{eq:Xsquared}
    \psi(h)^2=X_{\tau+h}^{\top}\left (a_0(h)+a_1(h)\right ) \leq \psi(h) \vert a_0(h) \vert + X_{\tau+h}^{\top}a_1(h),  
\end{align} where we simply applied Cauchy-Schwarz for the inequality. An application of Lemma~\ref{L:ao2} below yields that $\vert a_0(h)\vert\leq 1$. Moreover, by the definition of $\Omega^{\epsilon'}$ and the fact that $K$ is nonnegative, it follows that \begin{equation*}
    \int_{\tau}^{\tau+h}K(\tau+h-s)X_{\tau+h}^{\top}b^n(X_s)ds \leq 0.
\end{equation*} 
Therefore we can conclude \begin{equation*}
    \psi(h)^2\le \psi(h) \vert a_0(h) \vert + X_{\tau+h}^{\top}a_1(h) \leq \psi(h),
\end{equation*} which readily yields $\psi(h) \leq 1$, and thus $X_{\tau+h}\in\mathcal{B}$ for all $h\in(0,\epsilon')$ on $\Omega^{\epsilon'}$. On the other hand, by the definition of $\tau$ and continuity of $X$, there exists $h \in (0,\epsilon')$ such that $X_{\tau+h}\notin \mathcal{B}$ on $\Omega^{\epsilon'}$. But this readily shows $\mathbb{P}(\Omega^{\epsilon'})=0$ for all $\epsilon'$. On the other hand, by \eqref{eq: newcoeffcond} and continuity of $X$, it follows that $$\mathbb{P}(\tau<\infty) = \mathbb{P} \left (\bigcup_{\epsilon'\in (0,\epsilon)\cap \mathbb{Q}} \Omega^{\epsilon'}\right)=0.$$ 
\end{proof} The two following lemmas where used in the proof. 

\begin{lemma} Let $K$ satisfy \eqref{ass:K orthant}. Then, 
    \begin{equation} \label{eq:orth:1}
\text{$(\Delta_h K * L)(t)$ is nondecreasing in $t$,}
\end{equation}
as well as \begin{equation} \label{eq:orth:2}
0 \le (\Delta_hK * L)(t) \le (K * L)(t) = 1.
\end{equation}
\end{lemma}

\begin{proof}
    The proof of  \eqref{eq:orth:1} and  \eqref{eq:orth:2} can be found in the proof of \cite[Theorem 3.6]{abi2019affine}, see (3.9) and (3.10) therein. 
\end{proof}

\begin{lemma}\label{L:ao2} Let $K$ satisfy \eqref{ass:K orthant} and $f:[0,t]\to \mathcal B$ be continuous such that $f(t) \in \partial \mathcal{B}$ for some $t$. Then, the quantity
    \begin{align*}
        a_0(h) =\left( 1-(\Delta_h K * L)(t)\right)f(0) + (\Delta_h K* L)(0)f({t})  + (d(\Delta_h K * L)*f)(t), \quad h \geq 0, 
    \end{align*}
    satisfies 
    $$ a_0(h)^\top a_0(h) \leq 1, \quad h \geq 0. $$
\end{lemma}

\begin{proof}
    We first notice that, thanks to the Cauchy-Schwarz inequality, we have
    \begin{align}\label{eq:CSball}
        x^\top y \leq \sqrt{x^\top x}{\sqrt{y^\top y}} \leq 1, \quad x,y \in \mathcal B.  
    \end{align}
    We expand  
    \begin{align*}
        a_0(h)^\top a_0(h) &=   \left( 1-(\Delta_h K * L)(t)\right)^2 f(0)^\top f(0)  + (\Delta_h K* L)(0)^2f({t})^\top f(t)  +  \\
        &\quad + (d(\Delta_h K * L)*f)(t)^\top (d(\Delta_h K * L)*f)(t) \\
        &+  2  \left( 1-(\Delta_h K * L)(t)\right) (\Delta_h K* L)(0) f(0)^\top f(t) \\
        &+ 2  \left( 1-(\Delta_h K * L)(t)\right)f(0)^\top   (d(\Delta_h K * L)*f)(t)  \\
        &+ 2 (\Delta_h K* L)(0) f(t)^\top (d(\Delta_h K * L)*f)(t) \\
       & \leq  \left( 1-(\Delta_h K * L)(t)\right)^2  + (\Delta_h K* L)(0)^2  +  \\
        &\quad + (d(\Delta_h K * L)*1)^2(t) \\
        &+  2  \left( 1-(\Delta_h K * L)(t)\right) (\Delta_h K* L)(0) \\
        &+ 2  \left( 1-(\Delta_h K * L)(t)\right)   (d(\Delta_h K * L)*1)(t)  \\
        &+ 2 (\Delta_h K* L)(0) (d(\Delta_h K * L)*1)(t)\\
        &= \left( \left( 1-(\Delta_h K * L)(t)\right)  + (\Delta_h K* L)(0) +  (d(\Delta_h K * L)*1)(t)  \right)^2 
    \end{align*}
    where for the second inequality, we used \eqref{eq:CSball} combined with \eqref{eq:orth:1} and \eqref{eq:orth:2}. Finally, observing that 
    $$ (d(\Delta_h K * L)*1)(t)  =  (\Delta_h K * L)(t) - (\Delta_h K * L)(0)$$ yields the desired claim. 
\end{proof}

\appendix
\section{Well-posedness for Volterra-type integral equations}
\label{app:inteq}
In this section we study existence and uniqueness of solutions to a certain class of integral equations, including all the equations appearing in Section \ref{sec:Moments} for the characterization of moments. Fix $N,D\in \mathbb{N}$ and $T\geq 0$, and consider the domain $\mathcal{D}_T^{(N)}$ as in \eqref{eq:timedomain}. Denote by $C(\mathcal{D}_T^{(N)},\mathbb{R}^D)$ the space of continuous, vector-valued functions on the compact set $\mathcal{D}_T^{(N)}$. Moreover, consider the families of kernels $\{\mathcal{A}^{n,i}_{1}\}_{1\leq n \leq N;1\leq i\leq 2}$, and $\{\mathcal{A}^{n,m,j}_2\}_{1 \leq n,m \leq N;1\leq j\leq 3}$, such that $$\mathcal{A}^{n,i}_1:[0,T] \rightarrow \mathbb{R}^{D\times D}, \, \mathcal{A}^{n,m,j}_2:[0,T]\times [0,T] \rightarrow \mathbb{R}^{D\times D}, \, 1 \leq n,m \leq N,1\leq i\leq 2,1\leq j\leq 3.$$ We will make the following integrability assumptions for the matrix-norm of the kernels: 
\begin{equation}\label{eq:integrabilityforA}
    | \mathcal{A}^{n,i}_1 | \in L^1([0,T]), \quad  | \mathcal{A}_2^{n,m,j}(t,s) |\leq | \mathcal{G}_1^{n,m}(t) | \cdot | \mathcal{G}_2^{n,m}(s) |, 
\end{equation} 
for some $\mathcal{G}^{n,m}_1,\mathcal{G}^{n,m}_2 \in L^2$([0,T]), for all $1\leq n,m \leq N$, $1\leq i\leq 2$, $1\leq j\leq 3$.  Finally, for all $\mathbf{f} \in C(\mathcal{D}_T^{(N)},\mathbb{R}^D)$ we define 
\begin{equation}\label{def:operatorintegralequation}
\begin{aligned}
    (\Psi\mathbf{f})(t,T_1,\dots,T_N) &=  \sum_{1\leq n \leq N}\int_0^t\mathcal{A}^{n,1}_1(T_n-r)\mathbf{f}(r,r,(T_k)_{k \neq n})dr \\ 
    & \qquad\sum_{1\leq n \leq N}\int_0^t\mathcal{A}^{n,2}_1(T_n-r)\mathbf{f}(r,(T_k)_{k \neq n},r)dr\\
    & \qquad+ \sum_{1\leq n<m \leq N}\int_0^t\mathcal{A}^{n,m,1}_2(T_n-r,T_m-r)\mathbf{f}(r,r,r,(T_k)_{k \neq n,m})dr\\
    & \qquad+ \sum_{1\leq n<m \leq N}\int_0^t\mathcal{A}^{n,m,2}_2(T_n-r,T_m-r)\mathbf{f}(r,r,(T_k)_{k \neq n,m},r)dr\\
    & \qquad+ \sum_{1\leq n<m \leq N}\int_0^t\mathcal{A}^{n,m,3}_2(T_n-r,T_m-r)\mathbf{f}(r,(T_k)_{k \neq n,m},r,r)dr.
\end{aligned}
\end{equation}
\begin{proposition}\label{lemma:generalwellposedness}
Suppose that \eqref{eq:integrabilityforA} holds. Then, for any function $\mathbf{f}_0\in C([0,T]^N, \mathbb{R}^D)$ there exists a unique solution $\mathbf{f} \in C(\mathcal{D}_T^{(N)},\mathbb{R}^D)$ to the integral equation 
\begin{equation}\label{eq:generalintegralequation}
    \mathbf{f}(t,T_1,\dots,T_N) = \mathbf{f}_0(T_1,\dots,T_N) + (\Psi \mathbf{f})(t,T_1,\dots,T_N),\quad (t,T_1,\dots,T_N)\in\mathcal{D}_T^{(N)}.
\end{equation}
\end{proposition}
\begin{proof} 
For $\lambda \geq 0$ we introduce the norm $\Vert \mathbf{f} \Vert_{\lambda,\infty}= \sup_{(t,T_1,\dots,T_N) \in \mathcal{D}_T^{(N)}}e^{-\lambda t} | \mathbf{f}(t,T_1,\dots,T_N)|$, where $|\cdot |$ denotes the Euclidean norm on $\mathbb{R}^D$. One can check that $\Vert \cdot \Vert_{\lambda,\infty}$ is equivalent to the sup-norm, and $(C(\mathcal{D}_T^{(N)},\mathbb{R}^D),\Vert \cdot \Vert_{\lambda,\infty})$ is a Banach space. Following a standard proof pattern, we wish to prove that the operator 
$$T \mathbf{f} = \mathbf{f}_0(\cdot) + (\Psi \mathbf{f})(\cdot), \quad \mathbf{f}\in C(\mathcal{D}_T^{(N)},\mathbb{R}^D),$$ 
is invariant and contracts. For the invariance, we only need to show that $(t,T_1,\dots,T_N) \mapsto (T\mathbf{f})(t,T_1,\dots,T_N)$ is continuous. By assumption, $\mathbf{f}_0$ is continuous, and for $x=(t,T_1,\dots,T_N)$ and  $x' = (t',T_1',\dots,T_N')$, with $t \leq t'$, we have 
\begin{align*}
    & \left |(\Psi \mathbf{f})(t,T_1,\dots,T_N)-(\Psi \mathbf{f})(t',T'_1,\dots,T'_N)\right |  \leq \left |(\Psi\mathbf{f})(t,T_1,\dots,T_N)-(\Psi \mathbf{f})(t,T'_1,\dots,T'_N)\right |\\ & \quad + \left |(\Psi \mathbf{f})(t',T'_1,\dots,T'_N)-(\Psi \mathbf{f})(t,T'_1,\dots,T'_N)\right |. 
\end{align*} 
The second term converges to zero as $t \to t'$ by dominated convergence and \eqref{eq:integrabilityforA}. For the first term we only analyze the first and third components of $\Psi \mathbf{f}$ as the others behave similarly. Notice that for any $1\leq n,m\leq N$
\begin{align*}
 \int_0^t | & \mathcal{A}^{n,1}_1(T_n-r)\mathbf{f}(r,r,(T_k)_{k \neq n}) - \mathcal{A}^{n,1}_1(T_n'-r)\mathbf{f}(r,r,(T_k')_{k \neq n}) | dr \\ & \leq \int_0^t | \mathcal{A}^n_1(T_n-r) | \cdot | \mathbf{f}(r,r,(T_k)_{k\neq n})-\mathbf{f}(r,r,(T'_k)_{k\neq n}) | dr  \\ & \quad +\int_0^t | \mathcal{A}^n_1(T_n-r)-\mathcal{A}^n_1(T'_n-r) | \cdot | \mathbf{f}(r,r,(T'_k)_{k\neq n})| dr .
\end{align*} 
and
\begin{align*}
 \int_0^t | & \mathcal{A}^{n,m,1}_2(T_n-r,T_m-r)\mathbf{f}(r,r,r,(T_k)_{k \neq n,m}) - \mathcal{A}^{n,m,1}_2(T_n'-r,T_m'-r)\mathbf{f}(r,r,r,(T_k')_{k \neq n,m})| dr \\ & \leq \int_0^t | \mathcal{A}^{n,m,1}_2(T_n-r,T_m-r) | \cdot | \mathbf{f}(r,r,r,(T_k)_{k\neq n,m})-\mathbf{f}(r,r,r,(T'_k)_{k\neq n,m}) | dr \\  &\quad +\int_0^t | \mathcal{A}^{n,m,1}_2(T_n-r,T_m-r)-\mathcal{A}^{n,m}_2(T'_n-r,T'_m-r) | \cdot | \mathbf{f}(r,r,r,(T'_k)_{k\neq n,m}) | dr.
\end{align*}
In the expressions on the right of the two previous inequalities, the first integrals converge to zero as $x \to x'$ by \eqref{eq:integrabilityforA}, dominated convergence and the fact that $\mathbf{f}$ is continuous. For the second integrals, we can take $\Vert \mathbf{f} \Vert_{\infty}$ out of the integral, and then use \eqref{eq:integrabilityforA} and the fact that the translation of $L^p$-functions is continuous, see \cite[Proposition 1.6.13]{tao2011introduction}, to show that both terms converge to zero as $x'\to x$. Applying similar arguments to the other components of $\Psi\mathbf{f}$, we conclude that $T\mathbf{f} \in C(\mathcal{D}_T^{(N)},\mathbb{R}^D)$. \\  
Finally, denoting by $d_{\lambda}$ the metric induced by the norm $\Vert \cdot \Vert_{\lambda,\infty}$, we want to show that $T$ contracts for $\lambda$ large enough, that is 
\begin{equation}\label{eq:Tcontracts}
d_{\lambda}(T\mathbf{f},T\mathbf{g}) \leq q d_{\lambda}(\mathbf{f},\mathbf{g}), \quad \forall \mathbf{f},\mathbf{g} \in C(\mathcal{D}_T^{(N)},\mathbb{R}^D)
\end{equation}
for some $q < 1$ and $\lambda > 0$. By definition  \begin{equation}\label{eq:contractionineq}\begin{split}
& e^{-\lambda t}\left | (T\mathbf{f})(t,T_1,\dots,T_N)-(T\mathbf{g})(t,T_1,\dots,T_N) \right | \\ & \leq \sum_{1\leq n \leq N}\int_0^t e^{-\lambda t}|  \mathcal{A}^{n,1}_1(T_n-r) | | \mathbf{f}(r,r,(T_k)_{k\neq n})-\mathbf{g}(r,r,(T_k)_{k\neq n})| dr\\ & \quad + \sum_{1\leq n \leq N}\int_0^t e^{-\lambda t}|  \mathcal{A}^{n,2}_1(T_n-r) | | \mathbf{f}(r,(T_k)_{k\neq n},r)-\mathbf{g}(r,(T_k)_{k\neq n},r)| dr \\ & \quad + \sum_{1\leq n<m \leq N}\int_0^t e^{-\lambda t}|  \mathcal{A}^{n,m,1}_2(T_n-r,T_m-r) | | \mathbf{f}(r,r,r,(T_k)_{k\neq n,m})-\mathbf{g}(r,r,r,(T_k)_{k\neq n,m})| dr\\ & \quad + \sum_{1\leq n<m \leq N}\int_0^t e^{-\lambda t}|  \mathcal{A}^{n,m,2}_2(T_n-r,T_m-r) | | \mathbf{f}(r,r,(T_k)_{k\neq n,m},r)-\mathbf{g}(r,r,(T_k)_{k\neq n,m},r)| dr\\ & \quad + \sum_{1\leq n<m \leq N}\int_0^t e^{-\lambda t}|  \mathcal{A}^{n,m,3}_2(T_n-r,T_m-r) | | \mathbf{f}(r,(T_k)_{k\neq n,m},r,r)-\mathbf{g}(r,(T_k)_{k\neq n,m},r,r)| dr \\ & \leq d_{\lambda}(\mathbf{f},\mathbf{g})\times \\
&\quad\underbrace{\left ( \sum_{\substack{1\leq n \leq N\\i=1,2}} \int_0^te^{-\lambda(t-r)} | \mathcal{A}^{n,i}_1(T_n-r) | dr + \sum_{\substack{1\leq n<m \leq N\\j=1,2,3}} \int_0^t e^{-\lambda(t-r)} | \mathcal{A}^{n,m,j}_2(T_n-r,T_m-r) | dr\right )}_{C_{\lambda}(t,T_1,\ldots,T_N)}.
\end{split}\end{equation} We claim that for all $1\leq n,m \leq N$, $1\leq i\leq 2$, and $1\leq j\leq 3$,  we have 
\begin{equation}\label{eq:cont_const_tozero}
\sup_{0 \leq t \leq T_1 \leq T}\int_0^te^{-\lambda(t-r)} | \mathcal{A}^{n,i}_1(T_1-r) | dr, \sup_{0 \leq t \leq T_1,T_2 \leq T}\int_0^te^{-\lambda(t-r)} | \mathcal{A}^{n,m,j}_2(T_1-r,T_2-r) | dr \xrightarrow{\lambda \to \infty} 0.
\end{equation}
Indeed, using a change of variable $s = t-r$, for any $\delta > 0$ we can write the first integral as \begin{align*}
\int_0^te^{-\lambda s} | \mathcal{A}^{n,i}_1(T_1-t+s) | ds & \leq \int_0^{\delta}e^{-\lambda s} | \mathcal{A}^{n,i}_1(T_1-t+s) | ds+ \int_{\delta}^{t\vee\delta} e^{-\lambda s}| \mathcal{A}^{n,i}_1(T_1-t+s)| ds \\ & \leq \Vert \mathcal{A}^{n,i}_1 \Vert_{L^1([T_1-t,T_1-t+\delta])} + e^{-\lambda \delta} \Vert \mathcal{A}^{n,i}_{1} \Vert_{L^1([0,T])}.
\end{align*} Since $\mathcal{A}^{n,i}_1\in L^1([0,T])$ by assumption \eqref{eq:integrabilityforA}, for any $\epsilon > 0$, we can choose $\delta$ small enough, such that $\Vert \mathcal{A}^{n,i}_1 \Vert_{L^1([T_1-t,T_1-t+\delta])}\leq \epsilon/2$, uniformly in $(t,T_1)$. Choosing $\lambda>0$ such that $e^{-\lambda \delta} \Vert \mathcal{A}^{n,i}_{1} \Vert_{L^1([0,T])} \leq \epsilon / 2$, we obtain  $$\sup_{0 \leq t \leq T_1 \leq T}\int_0^te^{-\lambda s} | \mathcal{A}^{n,i}_1(T_1-t+s) | ds \leq \epsilon.$$ Similar considerations can be applied for the second term using the assumptions in \eqref{eq:integrabilityforA}. Indeed, for some $\mathcal{G}^{n,m}_1,\mathcal{G}^{n,m}_2\in L^2([0,T])$ we have $$\int_0^te^{-\lambda s} | \mathcal{A}^{n,m,j}_2(T_1-t+s,T_2-t+s) | ds \leq \max_{i\in \{1,2\}}\Vert \mathcal{G}^{n,m}_i \Vert_{L^2([T_i-t,T_i-t+\delta])}^2 + e^{-\lambda \delta} \max_{i\in \{1,2\}}\Vert \mathcal{G}^{n,m}_i \Vert_{L^2([0,T])}^2,$$ and we conclude with the same arguments as before, that we can find $\lambda$ large enough such that  $$ \sup_{0 \leq t \leq T_1,T_2 \leq T}\int_0^te^{-\lambda(t-r)} | \mathcal{A}^{n,m,j}_2(T_1-r,T_2-r) | dr \leq \epsilon.$$ Hence, \eqref{eq:cont_const_tozero} holds and we can choose $\lambda$ large enough, such that  $$\sup_{(t,T_1,\dots,T_N) \in \mathcal{D}_T^{(N)}}C_{\lambda}(t,T_1,\ldots,T_N)<1,$$ with $C_{\lambda}$ as in \eqref{eq:contractionineq}. This implies \eqref{eq:Tcontracts} for some $q<1$ and $\lambda > 0$. Banach fixed point theorem yields the conclusion.
\end{proof}

\begin{corollary}\label{cor:existenceumoments} The function $\mathbf{m}$ defined in Theorem \ref{th:momentformula} is the unique solution in $\mathcal{X}_T^{(N)}$ of \eqref{eq:momentformula} with initial condition $\mathbf{m}(0,T_1,\ldots,T_N)$, and $\mathbf{m}$ is continuous on $\mathcal{D}_T^{(N)}$.
\end{corollary}
\begin{proof}
We start by observing that $\mathcal{M}_T^{(N)}$ can be defined for bounded functions $f$ on $\mathcal{D}_T^{(N)}$ by inserting $r$ (one or multiple times) as the last arguments of the function $\mathbf{f}$ in the first, third and fourth addends in \eqref{eq:operatorM}. Furthermore, this extension of the operator $\mathcal{M}_T^{(N)}$ restricted to the subspace $C(\mathcal{D}_T^{(N)},\mathbb{R}^D)$ has the structure of the operator $\Psi$ in \eqref{def:operatorintegralequation} with $D=D_N$ (the cardinality of the set $\mathcal{I}^{(N)}$ in \eqref{eq:setI_N}), and the hypothesis \eqref{eq:integrabilityforA} holds. In addition, since $g_0$ is continuous then $\mathbf{m}(0,T_1,\ldots,T_N)$ is continuous as well. By Proposition \ref{lemma:generalwellposedness} there exists a solution $\mathbf{m}'\in C(\mathcal{D}_T^{(N)},\mathbb{R}^D)$ to the equation \eqref{eq:momentformula} with initial condition $\mathbf{m}(0,T_1,\ldots,T_N)$, where instead of $\mathcal{M}_T^{(N)}$ we consider the above mentioned extension of the operator. Furthermore, the proof of \eqref{eq:contractionineq} in Proposition \ref{lemma:generalwellposedness} also holds for bounded functions $\mathbf{f},\mathbf{g}$. Hence, uniqueness also holds over the space of bounded functions on $\mathcal{D}_T^{(N)}$. These observations imply uniqueness of the solutions in $\mathcal{X}_T^{(N)}$ of \eqref{eq:momentformula} with initial condition $\mathbf{m}(0,T_1,\ldots,T_N)$, and that the function $\mathbf{m}$ defined in Theorem \ref{th:momentformula} has to coincide with $\mathbf{m}'$ and in particular it is continuous.
\end{proof}
\begin{corollary}\label{cor:eqncoeffs}
Fix $\beta \in \mathbb{N}_0^d$ and $p\in\mathbb{N}$ such that $|\beta|\leq p$. Suppose that that for each $1\leq q<p$ and any $w\in\{1,\ldots,d\}^{q}$, $C^{(q)}_{\beta}(\cdot,\ldots,\cdot;w)$ is continuous on $\mathcal{D}_T^{(q)}$. Consider the system of equations \eqref{eq:coefficientformulas} in Theorem \ref{thm:momentsarepolynomials} -- seen as a system indexed over the elements $w\in\{1,\ldots,d\}^p$ -- with initial condition $\mathbf{f}_{\beta}^{(p)}(\cdots;w)$, $w\in\{1,\ldots,d\}^p$, as in \eqref{eq:coefficientformulas_initialcond}. Then for each $w\in\{1,\ldots,d\}^p$, $\mathbf{f}_{\beta}^{(p)}(\cdots;w)$ is in $C(\mathcal{D}_T^{(p)},\mathbb{R})$ and \eqref{eq:coefficientformulas} has a unique solution such that $C_{\beta}^{(p)}(\cdot,\ldots,\cdot;w)\in C(\mathcal{D}_T^{(p)},\mathbb{R})$.
\end{corollary}
\begin{proof} The continuity of $\mathbf{f}_{\beta}^{(p)}(\cdots;w)$ can be shown in the same way that we showed in the proof of Proposition \ref{lemma:generalwellposedness} that $T\mathbf{f}$ was continuous for $\mathbf{f}$ continuous. By taking $D=d^p$, i.e. the cardinality of $\{1,\ldots,d\}^p$, and $N=p$, equation \eqref{eq:coefficientformulas} has the structure of the equation in Proposition \ref{lemma:generalwellposedness} and hence the claim follows.
\end{proof}

\bibliographystyle{abbrv}
\bibliography{bibl}

\begin{thebibliography}{10}

\bibitem{abi2022laplace}
E.~Abi~Jaber.
\newblock The {L}aplace transform of the integrated {V}olterra {W}ishart
  process.
\newblock {\em Mathematical Finance}, 32(1):309--348, 2022.

\bibitem{abi2021weak}
E.~Abi~Jaber, C.~Cuchiero, M.~Larsson, and S.~Pulido.
\newblock A weak solution theory for stochastic {V}olterra equations of
  convolution type.
\newblock {\em The Annals of Applied Probability}, 31(6):2924--2952, 2021.

\bibitem{abi2019markovian}
E.~Abi~Jaber and O.~El~Euch.
\newblock Markovian structure of the {V}olterra {H}eston model.
\newblock {\em Statistics \& Probability Letters}, 149:63--72, 2019.

\bibitem{abi2019multifactor}
E.~Abi~Jaber and O.~El~Euch.
\newblock Multifactor approximation of rough volatility models.
\newblock {\em SIAM journal on financial mathematics}, 10(2):309--349, 2019.

\bibitem{abijaber2022joint}
E.~Abi~Jaber, C.~Illand, and S.~Li.
\newblock Joint {SPX-VIX} calibration with {G}aussian polynomial volatility
  models: deep pricing with quantization hints.
\newblock {\em arXiv preprint arXiv:2212.08297}, 2022.

\bibitem{abi2019affine}
E.~Abi~Jaber, M.~Larsson, and S.~Pulido.
\newblock Affine {V}olterra processes.
\newblock {\em The Annals of Applied Probability}, 29(5):3155--3200, 2019.

\bibitem{abi2024volatility}
E.~Abi~Jaber and S.~Li.
\newblock Volatility models in practice: Rough, path-dependent or markovian?
\newblock {\em Path-Dependent or Markovian}, 2024.

\bibitem{ackerer2018jacobi}
D.~Ackerer, D.~Filipovi{\'c}, and S.~Pulido.
\newblock The {J}acobi stochastic volatility model.
\newblock {\em Finance and Stochastics}, 22:667--700, 2018.

\bibitem{ahdida2013mean}
A.~Ahdida and A.~Alfonsi.
\newblock A mean-reverting {SDE} on correlation matrices.
\newblock {\em Stochastic Processes and their Applications}, 123(4):1472--1520,
  2013.

\bibitem{alfonsi2023nonnegativity}
A.~Alfonsi.
\newblock Nonnegativity preserving convolution kernels. {A}pplication to
  {S}tochastic {V}olterra {E}quations in closed convex domains and their
  approximation.
\newblock {\em arXiv preprint arXiv:2302.07758}, 2023.

\bibitem{barndorff2011ambit}
O.~E. Barndorff-Nielsen, F.~E. Benth, and A.~E. Veraart.
\newblock Ambit processes and stochastic partial differential equations.
\newblock {\em Advanced mathematical methods for finance}, pages 35--74, 2011.

\bibitem{BBV:13}
O.~E. Barndorff-Nielsen, F.~E. Benth, and A.~E.~D. Veraart.
\newblock {Modelling energy spot prices by volatility modulated Lévy-driven
  Volterra processes}.
\newblock {\em Bernoulli}, 19(3):803 -- 845, 2013.

\bibitem{BBV:14}
O.~E. Barndorff-Nielsen, F.~E. Benth, and A.~E.~D. Veraart.
\newblock {Modelling electricity futures by ambit fields}.
\newblock {\em Advances in Applied Probability}, 46(3):719 -- 745, 2014.

\bibitem{barndorff2007ambit}
O.~E. Barndorff-Nielsen and J.~Schmiegel.
\newblock Ambit processes; with applications to turbulence and tumour growth.
\newblock In {\em Stochastic Analysis and Applications: The Abel Symposium
  2005}, pages 93--124. Springer, 2007.

\bibitem{barndorff2008time}
O.~E. Barndorff-Nielsen and J.~Schmiegel.
\newblock Time change, volatility, and turbulence.
\newblock In {\em Mathematical Control Theory and Finance}, pages 29--53.
  Springer, 2008.

\bibitem{barndorff2009brownian}
O.~E. Barndorff-Nielsen and J.~Schmiegel.
\newblock Brownian semistationary processes and volatility/intermittency.
\newblock {\em Advanced financial modelling}, 8:1--26, 2009.

\bibitem{benth2020abstract}
F.~E. Benth, N.~Detering, and P.~Kr{\"u}hner.
\newblock Abstract polynomial processes.
\newblock {\em arXiv preprint arXiv:2010.02483}, 2020.

\bibitem{benth2021independent}
F.~E. Benth, N.~Detering, and P.~Kr{\"u}hner.
\newblock Independent increment processes: a multilinearity preserving
  property.
\newblock {\em Stochastics}, 93(6):803--832, 2021.

\bibitem{BK:14}
F.~E. Benth and P.~Kr{\"u}hner.
\newblock Representation of infinite-dimensional forward price models in
  commodity markets.
\newblock {\em Communications in Mathematics and Statistics}, 2(1):47--106,
  2014.

\bibitem{BK:15}
F.~E. Benth and P.~Kr\"uhner.
\newblock Derivatives pricing in energy markets: an infinite-dimensional
  approach.
\newblock {\em SIAM Journal on Financial Mathematics}, 6(1):825--869, 2015.

\bibitem{berger1980volterra}
M.~A. Berger and V.~J. Mizel.
\newblock Volterra equations with {I}t{\^o} integrals—{I}.
\newblock {\em The Journal of Integral Equations}, pages 187--245, 1980.

\bibitem{biagini2016polynomial}
F.~Biagini and Y.~Zhang.
\newblock Polynomial diffusion models for life insurance liabilities.
\newblock {\em Insurance: Mathematics and Economics}, 71:114--129, 2016.

\bibitem{BCKW:16}
J.~Blath, A.~G. Casanova, N.~Kurt, and M.~Wilke-Berenguer.
\newblock {A new coalescent for seed-bank models}.
\newblock {\em The Annals of Applied Probability}, 26(2):857 -- 891, 2016.

\bibitem{bondi2024affine}
A.~Bondi, G.~Livieri, and S.~Pulido.
\newblock Affine {V}olterra processes with jumps.
\newblock {\em Stochastic Processes and their Applications}, 168:104264, 2024.

\bibitem{cont2001empirical}
R.~Cont.
\newblock Empirical properties of asset returns: stylized facts and statistical
  issues.
\newblock {\em Quantitative finance}, 1(2):223, 2001.

\bibitem{cuchiero2019polynomial}
C.~Cuchiero.
\newblock Polynomial processes in stochastic portfolio theory.
\newblock {\em Stochastic processes and their applications}, 129(5):1829--1872,
  2019.

\bibitem{cuchiero2021measure}
C.~Cuchiero, F.~Guida, L.~di~Persio, and S.~Svaluto-Ferro.
\newblock Measure-valued affine and polynomial diffusions.
\newblock {\em Preprint ArXiv:2112.15129}, 2021.

\bibitem{cuchiero2012polynomial}
C.~Cuchiero, M.~Keller-Ressel, and J.~Teichmann.
\newblock Polynomial processes and their applications to mathematical finance.
\newblock {\em Finance and Stochastics}, 16:711--740, 2012.

\bibitem{CLS:18}
C.~Cuchiero, M.~Larsson, and S.~Svaluto-Ferro.
\newblock {Probability measure-valued polynomial diffusions}.
\newblock {\em Electronic Journal of Probability}, 24:1 -- 32, 2019.

\bibitem{cuchiero2021infinite}
C.~Cuchiero and S.~Svaluto-Ferro.
\newblock Infinite-dimensional polynomial processes.
\newblock {\em Finance and Stochastics}, 25(2):383--426, 2021.

\bibitem{cuchiero2023signature}
C.~Cuchiero, S.~Svaluto-Ferro, and J.~Teichmann.
\newblock Signature {SDE}s from an affine and polynomial perspective.
\newblock {\em Preprint arXiv:2302.01362}, 2023.

\bibitem{cuchiero2020generalized}
C.~Cuchiero and J.~Teichmann.
\newblock Generalized {F}eller processes and {M}arkovian lifts of stochastic
  {V}olterra processes: the affine case.
\newblock {\em Journal of evolution equations}, 20(4):1301--1348, 2020.

\bibitem{DZ:14}
G.~Da~Prato and J.~Zabczyk.
\newblock {\em Stochastic equations in infinite dimensions}.
\newblock Cambridge university press, 2014.

\bibitem{delemotte2023yet}
J.~Delemotte, S.~D. Marco, and F.~Segonne.
\newblock Yet another analysis of the sp500 at-the-money skew: Crossover of
  different power-law behaviours.
\newblock {\em Available at SSRN 4428407}, 2023.

\bibitem{DFS:03}
D.~Duffie, D.~Filipovi\'c, and W.~Schachermayer.
\newblock Affine processes and applications in finance.
\newblock {\em Ann. Appl. Probab.}, 13(3):984--1053, 2003.

\bibitem{ER:19}
O.~El~Euch and M.~Rosenbaum.
\newblock The characteristic function of rough {H}eston models.
\newblock {\em Mathematical Finance}, 29(1):3--38, 2019.

\bibitem{etheridge2000introduction}
A.~Etheridge.
\newblock {\em An introduction to superprocesses}.
\newblock Number~20. American Mathematical Soc., 2000.

\bibitem{filipovic2001consistency}
D.~Filipovic.
\newblock {\em Consistency problems for Heath-Jarrow-Morton interest rate
  models}.
\newblock Springer Science \& Business Media, 2001.

\bibitem{filipovic2016quadratic}
D.~Filipovi{\'c}, E.~Gourier, and L.~Mancini.
\newblock Quadratic variance swap models.
\newblock {\em Journal of Financial Economics}, 119(1):44--68, 2016.

\bibitem{filipovic2016polynomial}
D.~Filipovi{\'c} and M.~Larsson.
\newblock Polynomial diffusions and applications in finance.
\newblock {\em Finance and Stochastics}, 20(4):931--972, 2016.

\bibitem{filipovic2020markov}
D.~Filipovi{\'c}, M.~Larsson, and S.~Pulido.
\newblock Markov cubature rules for polynomial processes.
\newblock {\em Stochastic Processes and their Applications}, 130(4):1947--1971,
  2020.

\bibitem{GJR:22}
J.~Gatheral, T.~Jaisson, and M.~Rosenbaum.
\newblock Volatility is rough.
\newblock In {\em Commodities}, pages 659--690. Chapman and Hall/CRC, 2022.

\bibitem{GLS:90}
G.~Gripenberg, S.-O. Londen, and O.~Staffans.
\newblock {\em Volterra integral and functional equations}, volume~34 of {\em
  Encyclopedia of Mathematics and its Applications}.
\newblock Cambridge University Press, Cambridge, 1990.

\bibitem{guyon2022does}
J.~Guyon and M.~El~Amrani.
\newblock Does the term-structure of equity at-the-money skew really follow a
  power law?
\newblock {\em Available at SSRN 4174538}, 2022.

\bibitem{GL:23}
J.~Guyon and J.~Lekeufack.
\newblock Volatility is (mostly) path-dependent.
\newblock {\em Quantitative Finance}, 23(9):1221--1258, 2023.

\bibitem{hamaguchi2023weak}
Y.~Hamaguchi.
\newblock Weak well-posedness of stochastic {V}olterra equations with
  completely monotone kernels and non-degenerate noise.
\newblock {\em arXiv preprint arXiv:2310.16030}, 2023.

\bibitem{jacquier2019deep}
A.~J. Jacquier and M.~Oumgari.
\newblock Deep {PPDE}s for rough local stochastic volatility.
\newblock {\em Available at SSRN 3400035}, 2019.

\bibitem{kimura1964diffusion}
M.~Kimura.
\newblock Diffusion models in population genetics.
\newblock {\em Journal of Applied Probability}, 1(2):177--232, 1964.

\bibitem{LARSSON2017901}
M.~Larsson and S.~Pulido.
\newblock Polynomial diffusions on compact quadric sets.
\newblock {\em Stochastic Processes and their Applications}, 127(3):901--926,
  2017.

\bibitem{li2023measure}
Z.~Li.
\newblock {\em Measure-Valued Branching Markov Processes}.
\newblock Springer, 2022.

\bibitem{N:94}
I.~Norros.
\newblock A storage model with self-similar input.
\newblock {\em Queueing systems}, 16:387--396, 1994.

\bibitem{parent2022rough}
L.~Parent.
\newblock The rough path-dependent volatility model.
\newblock {\em Available at SSRN 4270481}, 2022.

\bibitem{promel2023stochastic}
D.~J. Pr{\"o}mel and D.~Scheffels.
\newblock Stochastic {V}olterra equations with {H}{\"o}lder diffusion
  coefficients.
\newblock {\em Stochastic Processes and their Applications}, 161:291--315,
  2023.

\bibitem{tao2011introduction}
T.~Tao.
\newblock {\em An introduction to measure theory}, volume 126.
\newblock American Mathematical Soc., 2011.

\bibitem{viens2019martingale}
F.~Viens and J.~Zhang.
\newblock A martingale approach for fractional {B}rownian motions and related
  path dependent {PDE}s.
\newblock {\em The Annals of Applied Probability}, 29(6):3489--3540, 2019.

\bibitem{wang2008existence}
Z.~Wang.
\newblock Existence and uniqueness of solutions to stochastic {V}olterra
  equations with singular kernels and non-{L}ipschitz coefficients.
\newblock {\em Statistics \& probability letters}, 78(9):1062--1071, 2008.

\bibitem{Willi}
W.~Willinger, M.~S. Taqqu, W.~E. Leland, and D.~V. Wilson.
\newblock Self-similarity in high-speed packet traffic: analysis and modeling
  of ethernet traffic measurements.
\newblock {\em Statistical science}, pages 67--85, 1995.

\end{thebibliography}
\end{document}